\documentclass[11pt]{amsart}
\usepackage{fix-cm}
\usepackage{graphicx}
\usepackage{amsmath}
\usepackage{amssymb}
\usepackage{amsopn}
\usepackage{txfonts}
\usepackage{mathrsfs}
\usepackage{enumerate}
\usepackage[hidelinks]{hyperref}
\usepackage{xcolor}
\usepackage{mathtools}

\DeclareMathOperator*{\Hess}{Hess}
\DeclareMathOperator*{\minpartial}{\partial^\circ}
\DeclareMathOperator*{\Dom}{Dom}
\DeclareMathOperator*{\clos}{clos}
\DeclareMathOperator*{\scm}{\textsc{sc}^{\,-}}

\DeclareMathOperator*{\conv}{conv}

\DeclareMathOperator*{\support}{supp}

\DeclareMathOperator{\divergence}{div}

\DeclareMathOperator*{\argmin}{argmin}

\DeclareMathOperator*{\gammaliminf}{\Gamma-\liminf}
\DeclareMathOperator*{\gammalimsup}{\Gamma-\limsup}
\DeclareMathOperator*{\gammalim}{\Gamma-\lim}
\DeclareMathOperator*{\silim}{\sigma-\lim}

\newcommand{\sconverge}{\xrightarrow[]{\sigma}}
\newcommand{\gconverge}{\xrightarrow[]{\text{G}}}
\newcommand{\gammaconverge}{\xrightarrow[]{\Gamma}}
\newcommand{\R}{\mathbb{R}}
\newcommand{\N}{\mathbb{N}}

\newcommand{\myquote}[1]{\begin{quotation}\emph{#1}\end{quotation}}

\theoremstyle{theorem}
\newtheorem{theorem}{Theorem}[section]

\newtheorem{proposition}{Proposition}[section]
\newtheorem{lemma}{Lemma}[section]
\newtheorem{definition}{Definition}[section]

\theoremstyle{definition}
\newtheorem{problem}{Problem}

\newtheorem{assumptions}{Set of Assumptions}

\theoremstyle{remark}

\newtheorem{remark}{Remark}[section]

\DeclareMathSizes{12}{12}{7}{5}
\calclayout
\title{Transport Energy}
\author{Enrico Facca}
\address{Enrico Facca, Scuola Normale Superiore di Pisa}
\email{enrico.facca@sns.it}

\author{Federico Piazzon}
\address{Federico Piazzon, Department of Mathematics \emph{Tullio Levi-Civita}, Universit\'a degli Studi di Padova}
\email{fpiazzon@math.unipd.it}


\begin{document}
\begin{abstract}
We introduce the \emph{transport energy} functional $\mathcal E$ (a variant of the Bouchitt\'e-Buttazzo-Seppecher shape optimization functional) and we prove that its unique minimizer is the optimal transport density $\mu^*$, i.e., the solution of Monge-Kantorovich equations. We study the gradient flow of $\mathcal E$ showing that $\mu^*$ is the unique global attractor of the flow.

We introduce a two parameter family $\{\mathcal E_{\lambda,\delta}\}_{\lambda,\delta>0}$ of strictly convex functionals approximating $\mathcal E$ and we prove the convergence of the minimizers $\mu_{\lambda,\delta}^*$ of $\mathcal E_{\lambda,\delta}$ to $\mu^*$ as we let $\delta\to 0^+$ and $\lambda\to 0^+.$ 

We derive an evolution system of fully non-linear PDEs as gradient flow of $\mathcal E_{\lambda,\delta}$ in $L^2$, showing existence and uniqueness of solutions. All the trajectories of the flow converge in $W^{1,p}_0$ to the unique minimizer $\mu_{\lambda,\delta}^*$ of $\mathcal E_{\lambda,\delta}.$   

Finally, we characterize $\mu_{\lambda,\delta}^*$ by a non-linear system of PDEs which is a perturbation of Monge-Kantorovich equations by means of a p-Laplacian.  
\end{abstract}
\maketitle
\tableofcontents

\section{Introduction}
\subsection{Optimal transport formulations and transport energy}
Optimal transport is a branch of mathematics that, intuitively, studies the problem of finding a least-cost strategy for moving a resource from one spatial distribution to a target one. The very first formulation of optimal transport was introduced by Monge in 1781. Nowadays it reads as follows. 
\begin{problem}[Monge]\label{problemMonge}
Let $\nu^+,\nu^-$ be two Borel measures on $\R^n$ with finite equal masses. Let $c:\R^n\times\R^n\rightarrow \R\cup\{+\infty\}$ be a Borel function. Find a Borel function $T:\R^n\rightarrow\R^n$ realizing the following infimum
$$\inf\left\{\int_{\R^n} c(x,T(x))d\nu^+, T_\#\nu^+ = \nu^-\right\},$$
where we denoted by $T_\#\nu^+$ the push-forward measure.
\end{problem}
The lack of compactness of the set of transport maps (e.g., Borel maps $T$ such that $T_\#\nu^+ = \nu^-$) leads to difficulties in finding solutions to Problem \ref{problemMonge}. For this reason, Kantorovich introduced 
the following relaxed formulation.
\begin{problem}[Kantorovich]\label{problemKantorovich}
Let $\nu^+,\nu^-$ be two Borel measures on $\R^n$ with finite equal masses. Let $c:\R^n\times\R^n\rightarrow \R\cup\{+\infty\}$ be a Borel function. 
Find a non-negative Borel Measure $\gamma$ on $\R^n\times \R^n$ realizing the following infimum
$$\inf\left\{\int_{\R^n\times \R^n} c(x,y)d\gamma(x,y) \right\},$$
under the constraints 
\begin{align*} 
&\gamma (A,\R^n) = \nu^{+}(A)\  \forall A \mbox{ Borel set in } \R^n,\\ 
&\gamma (\R^n,B) = \nu^{-}(B)\  \forall B \mbox{ Borel set in } \R^n.
\end{align*}
\end{problem}
In contrast to the case of Problem \ref{problemMonge}, a solution of Problem \ref{problemKantorovich} does exist under mild assumptions on $c$, e.g., lower semicontinuity and boundedness from below.
Optimal transport in the Kantorovich formulation has been studied by a number of authors in recent years (see, e.g., to \cite{Vi09,Sa14} and references therein for an extensive treatment of the subject). In the present work we focus on the case known as $L^1$ optimal transport, where
$$c(x,y):=|x-y|.$$
This setting reveals some difficulties, being the cost functional non-strictly convex. However, this line of research turns out to be very profitable, since Problem \ref{problemKantorovich} (possibly under further assumptions) can be re-casted in different equivalent formulations, \cite{Am03}. In particular, a PDE-based formulation was introduced by Evans and Gangbo in the seminal paper \cite{EvGa99}, their approach takes the following form.
\begin{problem}[Monge-Kantorovich equations] \label{ProblemEvGa}
Let $\Omega$ be a bounded convex Lipschitz domain of $\R^n$ and let $f=f^+-f^-\in L^\infty(\Omega)$ be a compactly supported function such that $\int_\Omega fdx=0.$
Find a non-negative function $\mu^*\in L^\infty(\Omega)$ for which the following system of PDEs admits a (non necessarily unique) weak solution $u^*$
\begin{equation}\label{EvGaPDE}
\begin{cases}
-\divergence{(\mu^*\nabla u^*)}=f,& \text{ in }\Omega\\
|\nabla u^*|\leq 1,& \text{ in }\Omega\\
|\nabla u^*|=1\,& \mu^*\text{ a.e. in }\Omega
\end{cases}.
\end{equation}
\end{problem}
Indeed in \cite{EvGa99} the authors proved that Problem \ref{ProblemEvGa} admits at least one solution. Later, Feldman and McCann showed \cite{FeMc02} the uniqueness of such solution $\mu^*$. We refer the reader to \cite{AmCaBrBuVi03} for  more complete results on existence and uniqueness.
\begin{definition}[Optimal transport density]
The unique solution $\mu^*$ of Problem \ref{ProblemEvGa} is termed \emph{optimal transport density}. 
\end{definition}
Under some additional regularity assumptions on the function $f$, starting from the solution $\mu^*$ of Problem \ref{ProblemEvGa}, the authors of \cite{EvGa99} were able to explicitly construct an optimal transport map for $\nu^\pm=f^\pm dx$ and the cost $c(x,y)=|x-y|$, namely a solution to Problem \ref{problemMonge}. The existence of an optimal transport map has been obtained via a different technique in \cite{AmPr03} for the case of absolutely continuous measures.

In \cite{FaCaPu18} the authors introduce the following fully non-linear system of evolution equations,
\begin{equation}\label{nonlinearode}
\begin{cases}
-\divergence(\mu(t,x)\nabla u(t,x))=f(x),& \text{ in }\Omega, t\geq 0\\
\mu(t,x)\nabla u(t,x)\cdot n(x),& x\in \partial\Omega, t\geq 0\\
\int_\Omega u(t,x)dx=0,& \forall t\geq 0\\
\frac d{dt} \mu(t,x)=\mu(t,x)|\nabla u(t,x)|-\mu(t,x),& x\in \Omega, t\geq 0\\
\mu(0,x)=\mu^0(x)>0,& x\in \Omega
\end{cases},
\end{equation}
and they conjecture that the long time asymptotics of its solution $\mu(t,\cdot)$ is precisely the optimal transport density $\mu^*$, regardless to the chosen Cauchy initial data $\mu^0$. They justify this claim by partial theoretical results. Indeed, they prove local (in time) existence and uniqueness of the trajectories in $\mathscr C^{0,\alpha}$ spaces, leaving their conjecture open, but still supported by numerical evidence. In addition, in \cite{FaDaCaPu18} a candidate Lyapunov functional (e.g., a functional decreasing along trajectories) for \eqref{nonlinearode} is provided. Starting from these ideas, in the present work we introduce the \emph{transport energy} $\mathcal E$ (see Definition \ref{Edefinition} below), a very minor modification of such candidate Lyapunov functional, and we study it under the following assumptions.
\begin{assumptions}
\begin{align}\label{fundamentalassupmtions}
&f=f^+-f^-\in L^\infty(\R^n),\;\;\int_{\R^n} f(x) dx=0,\notag\\
&S_f:=\support f\text{ is compact},\tag{H1}\\
&\Omega\text{ is a convex  bounded domain s.t. } \R^n\supset\Omega\supset \conv S_f.\notag  
\end{align}
\end{assumptions}
\begin{remark}
It is worth stressing that the role of $\Omega$ is not important here. Indeed in \cite{EvGa99} it is shown that any choice of $\Omega$ that strictly contains the convex envelope  $\conv S_f$ of the support of $f$ would lead to the same $\mu^*$, provided that the boundary of $\Omega$ is sufficiently away from $\conv S_f.$ It is not restrictive to assume $\Omega=B(0,R)$ for $R$ large enough.
\end{remark}

\begin{definition}[Transport energy]\label{Edefinition}
We denote by $\mathcal E: \mathcal M_+(\Omega)\rightarrow [0,+\infty]$ the transport energy functional defined by 
\begin{equation}\label{Edefinitioneq}
\mathcal E(\mu):=\sup_{u\in \mathscr C^1(\overline\Omega),\int_\Omega u dx=0}\left(2\int_\Omega fu\, dx-\int_\Omega |\nabla u|^2d\mu\right)+\int_\Omega d\mu.
\end{equation}
\end{definition}
Here and throughout the paper we denote by $\mathcal M(\Omega)$ the space of Borel signed measures on $\Omega,$ by $\mathcal M_+(\Omega)$ the non-negative Borel measures, and by $\mathcal M_1(\Omega)$ the space of Borel probability measures.

In the present work we aim at the solution and the variational approximation of the following problem.
\begin{problem}[Minimization of the transport energy]\label{problemEnricoMario}
Given $f,\Omega$ as in \eqref{fundamentalassupmtions}, find $\mu_{\mathcal E}\in \mathcal M_+(\Omega)$ such that
$$\mathcal E(\mu_{\mathcal E})=\inf_{\nu\in \mathcal M_+(\Omega)}\mathcal E(\nu).$$
\end{problem}

As we will state and prove in Proposition \ref{propositionexistenceanduniqueness}, the minimization of the functional $\mathcal E$ is closely related to the following variational problem first studied in \cite{BoBuSe97}; see also \cite{AmCaBrBuVi03}.
\begin{problem}[Bouchitt\'e-Buttazzo-Seppecher shape optimization] \label{problemBB}Given $m>0$, $\nu\in \mathcal M(\Omega)$, $\nu(\Omega)=0,$ and an open convex set $\Omega\subset \R^n,$ find $\mu_{\mathcal B}\in \mathcal M_+(\Omega)$, $\int_\Omega d\mu_{\mathcal B}=m$, that maximizes
\begin{equation}
\mathcal B(\mu):=\inf\left\{\int_\Omega |\nabla v|^2d\mu-2 \int_\Omega v d\nu,\;v\in \mathscr C^\infty(\overline\Omega)\right\}
\end{equation}
among all $\mu\in \mathcal M_+(\Omega)$ such that $\int_\Omega d\mu=m\}.$ 
\end{problem}
\subsection{Our results}
Solving Problem \ref{problemBB} under the Set of Assumptions 1 is equivalent, up to finding the correct value of the parameter $m$, to solving Problem \ref{ProblemEvGa}. Indeed, it has been shown (see \cite{FeMc02} for existence, \cite[Th. 5.2]{AmCaBrBuVi03} for uniqueness and regularity) that, if the measure $\nu$ is absolutely continuous with respect to the Lebesgue measure restricted to $\Omega$, with (positive and negative) densities $f^+,f^-\in L^s(\Omega)$, then there exists a unique solution $\mu_{\mathcal B}\in L^s$ to Problem \ref{problemBB}, and moreover, if $s=+\infty$,  we have
\begin{equation*}
\mu_{\mathcal B}=\frac{m}{\int_\Omega\mu^*dx}\mu^*.
\end{equation*} 
In contrast, the transport energy functional $\mathcal E$ has the desirable advantage of forcing the mass of its minimizers to be equal to $\int_\Omega \mu^*dx.$ More precisely, we prove in Section 2 (see Proposition \ref{propositionexistenceanduniqueness}) that:
\myquote{under the Set of Assumptions 1, the functional $\mathcal E$ has a unique minimizer $\mu_{\mathcal E},$ moreover $\mu_{\mathcal E}$ is an absolutely continuous measure with respect to the Lebesgue measure and its density is $\mu^*,$ namely the optimal transport density.}
\begin{remark}
In view of this result, from now on we use only the notation $\mu^*$, which is customary in the framework of optimal transport, both for the optimal transport density and for the density with respect to the Lebesgue measure of the unique minimizer of $\mathcal E$. At the same time, for notational convenience, we will use indifferently the symbol $\mu^*$ to identify both the density and the corresponding measure. The context will clarify the meaning.
\end{remark}

In \textbf{Section \ref{SecMinimizeE}} we characterize the solution of Problem \ref{problemEnricoMario} as long time asymptotics of the gradient flow of $\mathcal E.$  In the present work we address the study of the gradient flow of $\mathcal E$ in a \emph{purely metric framework}, see\cite{AmGiSa00}. The results on this subject, which are relevant for our purposes, are summarized in Appendix \ref{app1}. More precisely, in Section \ref{SecMinimizeE} we define a metric $d_w$ on $\mathcal M_+(\Omega)$ and we study the two main metric formulations of of the gradient flow of $\mathcal E$. Namely, we build the solution $\mu(t;\mu^0)$ of the \emph{evolution variational inequality} relative to $\mathcal E$, i.e., 
\begin{equation*}
\begin{cases}
\displaystyle\frac 1 2\frac{d}{dt}d_w^2(\mu(t;\mu^0),\nu)\leq \mathcal E(\nu)-\mathcal E(\mu(t;\mu^0))&,\;\text{ for a.e. }t\in[0,+\infty[,\;\forall\nu\in \mathcal M_+(\Omega)\\
\displaystyle\lim_{t\downarrow 0}d_w(\mu(t;\mu^0),\mu^0)=0& 
\end{cases},
\end{equation*}
and we show that the curves $t\mapsto\mu(t;\mu^0)$ are \emph{curves of maximal slope} for $\mathcal E$ (see appendix \ref{subsectionmetric}) that satisfy the \emph{energy identity}, i.e.,
\begin{equation*}
\mathcal E(\mu(t;\mu^0))=\mathcal E(\mu^0)-\int_0^t [|\partial \mathcal E|(\mu(s;\mu^0))]^2ds,\;\forall t>0.
\end{equation*}
Moreover we show (see Theorem \ref{theoremresultsforE} and \ref{easyresultforE}) that:
\myquote{for any $\mu^0\in \mathcal M_+(\Omega)$, the long time asymptotics in the weak$^*$ topology of the curve $t\mapsto\mu(t;\mu^0)$ is precisely $\mu^*.$}
In \textbf{Section \ref{SecApproxE}} we introduce a variational approximation of Problem \ref{problemEnricoMario}. Namely we define a two parameter family of strictly convex functionals $\{\mathcal E_{\lambda,\delta}\}_{\lambda,\delta>0}$ that can be thought of as  \emph{regularized} approximations of $\mathcal E.$ We study the $\Gamma$-limit (see Appendix \ref{app2} for a summary of the results employed in this work) of $\mathcal E_{\lambda,\delta}$ as $\delta\to 0^+$, $\lambda\to 0^+$ and we prove (see Theorem \ref{convergenceminimizerstheorem}) that
$$\gammalim_{\lambda\to 0^+}\gammalim_{\delta\to 0^+}\mathcal E_{\lambda,\delta}=\mathcal F_0, $$
where $\mathcal F_0$ is the relaxation (with respect to the weak$^*$ topology of $\mathcal M_+(\Omega)$) of the restriction of $\mathcal E$ to $W^{1,p}_0(\Omega)$, $p>n$. The functional $\mathcal F_0$ and $\mathcal E$ may be different.  However, using the regularity of $\mu^*,$ we can still prove that:
$$\mu^*=\argmin_{\mu\in\mathcal M_+(\Omega)}\mathcal E=\argmin_{\mu\in\mathcal M_+(\Omega)}\mathcal F_0$$
\myquote{and thus, the unique minimizers $\mu_{\lambda,\delta}^*$ of $\mathcal E_{\lambda,\delta}$ converge, with respect to $d_w$ and in the weak$^*$ topology of $\mathcal M_+(\Omega)$, to the optimal transport density $\mu^*,$ as $\lambda\to 0^+,$ $\delta\to 0^+,$ i.e.,
$$\lim_{\lambda\to 0^+}\lim_{\delta\to 0^+}d_w(\mu_{\lambda,\delta}^*,\mu^*)=0 .$$}
We also derive in Proposition \ref{propPDE} the following PDE-based characterization of $\mu_{\lambda,\delta}^*.$ 
\myquote{There exists a unique $u^*_{\lambda,\delta}\in W^{1,2}(\Omega)$, $\int_\Omega u^*_{\lambda,\delta}dx=0$ such that 
\begin{equation*}
\begin{cases}
1-|\nabla u_{\lambda,\delta}^*|^2-\delta p \Delta_p \mu_{\lambda,\delta}^*=0& \text{ on }\support \mu_{\lambda,\delta}^*\\
|\nabla u_{\lambda,\delta}^*|^2\leq 1& \text{ on }\{\mu_{\lambda,\delta}^*=0\}\\
-\divergence((\mu_{\lambda,\delta}^*+\lambda)\nabla u_{\lambda,\delta}^*)=f& \text{ in }\Omega\\
\mu_{\lambda,\delta}^*\geq 0& \text{ in }\Omega\\
\mu_{\lambda,\delta}^*=0& \text{ on }\partial \Omega\\
\partial_n u_{\lambda,\delta}^*=0& \text{ on } \partial\Omega\\
\int_\Omega u_{\lambda,\delta}^* dx=0
\end{cases}.
\end{equation*}}
Finally, in \textbf{Section \ref{SecMinimizeEld}} we study the dynamic minimization of the functionals $\mathcal E_{\lambda,\delta}$ acting on the Hilbert space $L^2(\Omega)$, for any $\lambda,\delta>0$, by the $L^2$ gradient flow 
\begin{equation*}
\begin{cases}
\frac d{dt}\mu(t,x)=[|\nabla u(t,x)|^2-1+\delta p\Delta_p\mu(t,x)]\chi_{\{\mu>0\}}+[(|\nabla u|^2-1)\chi_{\{\mu=0\}}]^+&\text{ in }[0,+\infty[\times\Omega\\
-\divergence((\mu(t,x)+\lambda)\nabla u(t,x))=f(x)&\text{ in }[0,+\infty[\times\Omega\\
\mu(t,x)=0&\text{ in }[0,+\infty[\times\partial\Omega\\
\partial_n u(t,x)=0,\int_\Omega u(t,x) dx=0&\text{ in }[0,+\infty[\times\partial\Omega\\
\mu(0,x)=\mu^0&\text{ for any }x\in \Omega
\end{cases}.
\end{equation*}
In this regularized and Hilbertian setting we can prove (see Theorem \ref{Thl2gf} and Theorem \ref{Thl2as})
\myquote{the existence and the uniqueness of the gradient flow and its convergence to the unique minimizer $\mu^*_{\lambda,\delta}$ of $\mathcal E_{\lambda,\delta}$, regardless the choice of the initial data. Therefore we have
$$\mu^*=\lim_{\lambda\to 0^+}\lim_{\delta\to 0^+}\lim_{t\to +\infty}\mu_{\lambda,\delta}(t;\mu^0),\;\;\forall \mu^0\in W^{1,p}_0(\Omega),\; \mu^0\geq 0.$$}
\subsection*{Acknowledgements}

We deeply thank Roberto Monti for many long and fruitful discussions that have been fundamental to achieve our results. Section \ref{SecMinimizeEld} has been essentially influenced by the discussions with Filippo Santambrogio, the authors would like to thank him.
\section{Equivalence of Problems \ref{ProblemEvGa}, \ref{problemEnricoMario} and \ref{problemBB} with $m=\int_\Omega \mu^* dx$}
\begin{proposition}\label{propositionexistenceanduniqueness}
Let us assume \eqref{fundamentalassupmtions}. Then Problem \ref{problemEnricoMario} has a unique solution $\mu_{\mathcal E}.$ Moreover $\mu_{\mathcal E}$ is absolutely continuous with respect to the Lebesgue measure, with density $\mu^*$, i.e.,
\begin{equation}\label{existenceanduniqueness}
\mu_{\mathcal E}(A)=\int_A\mu^*dx,\;\;\forall \text{Borel subset }A\subseteq \Omega.
\end{equation}
\end{proposition}
\begin{proof}[Proof of Proposition \ref{propositionexistenceanduniqueness}]
We introduce a shorter notation for the sake of readability. Let, $\forall \mu\in \mathcal M_+(\Omega),$
\begin{align*}
\mathcal L(\mu)&:= \sup_{u\in \mathscr C^1(\overline\Omega),\int_\Omega u dx=0}\left(2\int_\Omega fu\, dx-\int_\Omega |\nabla u|^2d\mu\right)  ,\\
M(\mu)&:=\int_\Omega d\mu.
\end{align*}
Let us show that, if 
$$\hat\mu\in \argmin_{\mathcal M_+(\Omega)} \mathcal E,$$
then we have
\begin{align}
&M(\hat\mu)=\mathcal L(\hat\mu)\;,\label{equalitymassenergy}\\
&\hat\mu\in\argmin_{\nu\in \mathcal M_+(\Omega):M(\nu)=\mathcal L(\nu)} \mathcal L(\nu)\;,\label{minimizeonvariety}\\
&\frac {\hat\mu}{M(\hat\mu)}\in \argmin_{\mathcal M_1(\Omega)}\mathcal L(\nu)\;.\label{minBB}
\end{align}
In order to show \eqref{equalitymassenergy}, we consider, for any $\mu\in \mathcal M_+(\Omega)$, $\mu\neq 0$, the function
$$\Phi_\mu(t):=\mathcal E(t\mu)=\mathcal L(t\mu)+M(t\mu),\;\;t>0.$$
Since $t\mapsto M(t\mu)$ is $1$ homogeneous and $t\mapsto\mathcal L(t\mu)$ is $(-1)$-homogeneous, we have
\begin{align*}
\Phi_\mu(t)&=\frac 1 t \mathcal L(\mu)+t M( \mu),\\
\Phi_\mu'(t)&=-\frac 1 {t^2} \mathcal L(\mu)+M( \mu),\\
\Phi_\mu''(t)&=\frac 2 {t^3} \mathcal L(\mu).
\end{align*}
In particular, being\footnote{This is a standard result. One possible proof is the following. Assume by contradiction that there exists a non-zero measure $\nu\in \mathcal M_+(\Omega)$ such that $\mathcal L(\nu)=0.$ Then we have $\mathcal E(t\nu)=M(t\nu)\to 0$ as $t\to 0.$ Notice that $\mathcal E$ is clearly lower semicontinuous with respect to the weak$^*$ convergence of measures, being defined by the supremum of continuous functionals. By the lower semicontinuity of $\mathcal E$ we have $\mathcal E(0)\leq \lim_{t\to 0^+}\mathcal E(t\nu)=0.$ On the other hand we can show that $\mathcal E(0)=+\infty.$ In fact, take $u_k=k\cdot f\ast\eta_k$ with $\eta_k$ a mollification kernel of step $1/k$, and note that $\mathcal E(0)=\mathcal L(0)\geq \mathcal L(0,u_k)=k\int_\Omega f \cdot f\ast\eta_k\to +\infty.$ Thus we have a contradiction and hence $\mathcal L(\mu)>0$ for any $\mu\in \mathcal M_+(\Omega).$} $\mathcal L(\mu)>0$ for any non-zero measure in $\mathcal M_+(\Omega)$, the function $\Phi_\mu$ is a strictly convex function, having the unique global minimum at 
$$t=t_\mu:=\frac{\sqrt{\mathcal L(\mu)}}{\sqrt{M(\mu)}}$$
with
$$\Phi_\mu(t_\mu)=2\sqrt{\mathcal L(\mu)}\sqrt{M(\mu)}.$$
Notice in particular that $  \mathcal L(t_\mu\mu)=M(t_\mu \mu),$ $\forall \mu\in \mathcal M_+(\Omega).$

We can conclude that $t_{\hat\mu}=1,$ that is equation \eqref{equalitymassenergy} holds. Indeed, assuming by contradiction  $t_{\hat \mu}\neq 1,$ we would have
$$\mathcal E(t_{\hat\mu}\hat\mu)=\Phi_{\hat\mu}(t_{\hat\mu})<\Phi_{\hat\mu}(1)=\mathcal E(\hat \mu)\leq \mathcal E(t_{\hat\mu}\hat\mu).$$

Equation \eqref{minimizeonvariety} can be proved similarly. Assume that we can find $\bar \mu\in \mathcal M_+(\Omega)$ such that $\mathcal L(\bar \mu)=M(\bar \mu)$ and $\mathcal L(\bar\mu)<\mathcal L(\hat\mu).$ Then, by using \eqref{equalitymassenergy},
$$\mathcal E(\bar\mu)<\mathcal E(\hat\mu)$$
contradicting the hypothesis $\hat\mu\in \argmin_{\mathcal M_+(\Omega)}\mathcal E.$ Thus \eqref{minimizeonvariety} must hold. 

In order to prove equation \eqref{minBB}, we pick any $\mu\in \mathcal M_1(\Omega)$ and notice that in such a case we have $t_\mu=\sqrt{\mathcal L(\mu)}.$ We can write
$$\mathcal L\left(\frac{\hat\mu}{M(\hat \mu)}  \right)=M(\hat \mu)\mathcal L(\hat\mu)=\left(\mathcal L(\hat\mu)\right)^2\leq \left(\mathcal L(t_\mu\mu)\right)^2=\mathcal L(\mu).$$
Here the first and the last equalities are due to the homogeneity of degree $-1$ of $\mathcal L$, the second equality to \eqref{equalitymassenergy}, and the inequality is due to \eqref{minimizeonvariety}. Therefore, using existence, uniqueness and regularity of Problem \ref{problemBB} with $m=1$, we have
$$\frac{\hat\mu}{M(\hat\mu)}=\mu_{\mathcal B}=\frac{\mu^*}{M(\mu^*)}.$$
This means that the set $\argmin \mathcal E$ consists, at most, of a one parameter family. However, the property 
$$2M(\hat\mu)=2\mathcal L(\hat\mu)=\mathcal E(\hat\mu),\;\;\forall \hat\mu\in \argmin \mathcal E,$$ 
reduces such a family to a single element that we denote by $\mu_{\mathcal E}.$

We are left to prove that $M(\mu^*)=M(\mu_\mathcal E).$ Notice however that, since
$$\left(M(\mu_{\mathcal E})\right)^2=\mathcal L\left(\frac{\mu_{\mathcal E}}{M(\mu_{\mathcal E})} \right)=\mathcal L\left(\frac{\mu^*}{M(\mu^*)} \right)=M(\mu^*)\mathcal L(\mu^*),$$
 it would suffice to prove 
\begin{equation}\label{fixthemass}
M(\mu^*)=\mathcal L(\mu^*)
\end{equation}
 and the proof of Proposition \ref{propositionexistenceanduniqueness} will be done.

The inequality $M(\mu^*)\leq \mathcal L(\mu^*)$ follows from
$$\mathcal L(\mu^*)\geq 2\int_\Omega f u^*dx-\int_\Omega |\nabla u^*|^2 d\mu=\int_\Omega |\nabla u^*|^2 d\mu=\int_\Omega\mu^*=M(\mu^*).$$
Here $u^*$ denotes any Monge-Kantorovich potential built as in  \cite{EvGa99}. Then the inequality is a consequence of the fact that $u^*$ is a competitor in the upper envelope defining $\mathcal L$ and the last three equalities follow from the defining properties of the pair $(\mu^*,u^*)$, i.e., equation \ref{EvGaPDE}.
To get the opposite inequality, we use the dual characterization (see \cite{BoBuSe97}) of Problem \ref{problemBB}, that is
\begin{align*}
\mathcal L(\mu^*)=&\sup_{\phi\in \mathscr C^1(\overline\Omega)}2\int_\Omega f\phi dx-\int_\Omega |\nabla \phi|^2d\mu^*=\inf_{\xi\in [L^2(\mu^*)]^n:\;\divergence(\mu^*\xi)=f}\int_\Omega |\xi|^2d\mu^*\\
\leq& \int_\Omega |\nabla u^*|^2d\mu^*=\int_\Omega d\mu^*=M(\mu^*).
\end{align*}
Here we used the same properties of $(\mu^*,u^*)$ as above. This last inequality concludes the proof of \eqref{fixthemass} and thus we proved that
$$\mu_\mathcal E=\mu^*,$$
which in particular implies the $L^\infty$ regularity of the minimizer $\mu_{\mathcal E}$; see \cite{AmCaBrBuVi03,AmGi13}.   
\end{proof}

\section{Dynamical minimization of $\mathcal E$}\label{SecMinimizeE}
In this section we aim at characterize $\mu^*$ as the long time asymptotics of the \emph{gradient flow} generated by the transport energy functional $\mathcal E$, e.g., the evolution system that formally writes as $\frac d{dt}\mu=-\nabla \mathcal E(\mu).$ This idea partially goes back to \cite{Fa18}, where formal computations relating \eqref{nonlinearode} and the gradient flow of $\mathcal E$ were presented. However, it is not immediate to find a natural ambient space for the rigorous study of the gradient flow equation for $\mathcal E$. For instance, if we state it in $L^\infty$, then we have to deal with the lack of reflexivity and separability of the chosen space. If instead we use the topology of $L^2$ we  loose the continuity and the differentiability properties of $\mathcal E.$ 

A different approach is to address the study  a \emph{purely metric} formulation of the gradient flow equation, following \cite{AmGiSa00}. In the present section we pursue this strategy. More precisely, we work in the space $(\mathcal M_+(\Omega),d_w),$ where $\mathcal M_+(\Omega)$ is the space of finite Borel measures and
\begin{equation}\label{d2def}
d_w(\mu,\nu):=\left(\sum_{k=0}^{+\infty} 2^{-k}\left|\int_\Omega \phi_kd\mu-\int_\Omega \phi_kd\nu\right|^2 \right)^{1/2},
\end{equation}
for a given sequence $\{\phi_k\}\subset \mathscr C^0(\overline\Omega)$ dense in the uniform norm unit sphere of $\mathscr C^0(\overline\Omega).$ In such a metric space we obtain (see Theorem \ref{easyresultforE} and  \ref{theoremresultsforE}) existence, uniqueness and long time asymptotics of \emph{curves of maximal slope} for $\mathcal E$ and of the solution of the corresponding \emph{evolution variational inequality}, two metric formulations of the gradient flow.

These existence and uniqueness results essentially rely on a useful geometric property of $d_w^2$, namely its $2$-convexity (in other words $(\mathcal M_+(\Omega),d_w)$ is non positively curved). For this reason we state and prove this convexity result first.
\begin{lemma}\label{lemma1convdw}
The function $d_w^2$ is $2$-convex, that is, $\forall \mu_0,\mu_1,\nu\in \mathcal M_+(\Omega)$ there exists a curve 
$\gamma:[0,1]\rightarrow \mathcal M_+(\Omega)$ with $\gamma(0)=\mu_0,\;\gamma(1)=\mu_1$ such that, for any $t\in (0,1)$ we have
\begin{equation}
d_w(\nu,\gamma(t))^2\leq(1-t) d_w(\nu,\mu_0)^2 +t d_w(\nu,\mu_1)^2-t(1-t) d_w(\mu_0,\mu_1)^2\label{conv},
\end{equation}
moreover we can pick $\gamma(t):=t\mu_1+(1-t)\mu_0$
 and obtain the equality case of \eqref{conv}.
 \end{lemma} 
\begin{proof}
The equation \eqref{conv} follows immediately choosing $\gamma(t):=(1-t)\mu_0+t\mu_1$ and using the identity
\small
\begin{align*}
&\left(\int_\Omega f_k d((1-t)\mu_0+t\mu_1-\nu)\right)^2\\
=&(1-t)^2\left(\int_\Omega f_k d(\mu_0- \nu)\right)^2\\
&\;\;\;\;\;\;\;\;\;\;\;\;\;\;+t^2\left(\int_\Omega f_k d(\mu_1- \nu)\right)^2+2t(1-t)\left(\int_\Omega f_k d(\mu_0- \nu)\right)\left(\int_\Omega f_k d(\mu_1- \nu)\right)\\
=&(1-t)\left(\int_\Omega f_k d(\mu_0- \nu)\right)^2+t \left(\int_\Omega f_k d(\mu_1- \nu)\right)^2-t(1-t)\left(\int_\Omega f_k d(\mu_0- \mu_1)\right)^2,
\end{align*}
\normalsize
multiplying by $2^{-k}$ and summing over $k=1,,2,\dots.$
\end{proof}

Let us set, $\forall c>0$
$$\mathscr M_c:=\{\mu\in \mathcal M_+(\Omega): \mathcal E(\mu)\leq c\}.$$

\begin{theorem}[Curves of maximal slope for $\mathcal E$]\label{easyresultforE}
Let $f,\Omega$ satisfy \eqref{fundamentalassupmtions}. Then, for any $c>\min_{\mathcal M_+(\Omega)}$ and any $\mu^0\in \mathscr M_c$, the class of minimizing movements $MM(\mu^0,\mathcal E,d_w)$ is not empty. Its elements are curves of $d_w$-maximal slope for $\mathcal E$ with respect to its \emph{strong} upper gradient $|\partial \mathcal E|$ and, for any such curve $t\mapsto\mu(t),$ we have
\begin{equation}\label{energyequalityE}
\mathcal E(\mu(t))=\mathcal E(\mu^0)-\int_0^t [|\partial \mathcal E|(\mu(s))]^2ds,\;\forall t>0.
\end{equation} 
\end{theorem}
\begin{proof}
Let us notice that the following properties hold.
\begin{align}
&d_w\text{ induces the weak* topology on }\mathscr M_c,\label{weakstartopology}\\
&\mathcal E\text{ is }d_w\text{-lower semicontinuous},\label{lscdw}\\
&\mathscr M_c\text{ is sequentially compact w.r.t. the metric }d_w.\label{dwcompactness}
\end{align}
Moreover $\forall \mu_0,\mu_1, \in \mathscr M_c$, $t\in (0,1),\tau>0$ we have
\small
\begin{align}
&\mathcal E(t\mu_1+(1-t)\mu_0)+\frac{d_w(\mu_0,t\mu_1+(1-t)\mu_0)^2}{2\tau}\label{convexity}\\
&\leq(1-t) \left(\frac{\mathcal E(\mu_0)}{2\tau}\right) +t\left( \mathcal E(\mu_1)+\frac{d_w(\mu_0,\mu_1)^2}{2\tau}\right)-t(1-t) \frac{d_w(\mu_0,\mu_1)^2}{2\tau}\notag
\end{align}
\normalsize
In particular \eqref{lscdw} follows from \eqref{weakstartopology} if we notice that $\mathcal E$ is the sum of a continuous functional with respect to the weak$^*$ topology and a supremum of continuous functional (with respect to the same topology) and hence $\mathcal E$ is weak$^*$ lower semicontinuous.

Also \eqref{dwcompactness} follows from \eqref{weakstartopology}. Indeed $\mathcal E(\mu)<c$ implies $\int_\Omega d\mu<c$ and the weak$^*$ topology of measures is well known to be sequentially compact on mass bounded subsets.

Being the sum of the linear functional $\mu\mapsto\int_\Omega d\mu$ and the supremum of among a family of affine functionals, the functional $\mathcal E$ is convex. The combination of the convexity of $\mathcal E$ with Lemma \ref{lemma1convdw} proves \eqref{convexity}, see Remark \ref{Ambrosioremark}.

The conclusions of Theorem \ref{easyresultforE} essentially follow by \cite[Th. 2.3.3, Cor. 2.4.10]{AmGiSa00}. This two results are recalled (together with some needed definitions) in Appendix \ref{app1} for reader's convenience; see Theorem \ref{Ambrosiotheorem} and Corollary \ref{Ambrosiolemma}, respectively.

More in detail, due to \cite[Cor. 2.4.10]{AmGiSa00} (see Lemma \ref{Ambrosiolemma}), $|\partial \mathcal E|$ is lower semicontinuous with respect to $d$, i.e.,
\begin{equation}\label{equalityofslopes}
|\partial\mathcal E|(\mu)=|\partial^-\mathcal E|(\mu):=\inf\{\liminf_k|\partial \mathcal E|(\mu_k), \sup\{d_w(\mu,\mu_k),\mathcal E(\mu_k)\}<+\infty\},
\end{equation}
and it is a strong upper gradient.
We can apply \cite[Th. 2.3.3]{AmGiSa00} (see Theorem \ref{Ambrosiotheorem}) due to this last two properties, to \eqref{lscdw}, and to \eqref{weakstartopology} . We obtain that any generalized minimizing movement $t\mapsto \mu(t)\in GMM(\mu^0,\mathcal E,d_w)$ is a curve of maximal slope, and the following energy equality holds
$$\mathcal E(\mu(t))=\mathcal E(\mu^0)-\frac 1 2\int_0^t [|\partial \mathcal E|(\mu(s))]^2ds-\frac 1 2\int_0^t |\mu'|^2(s)ds,\;\forall t>0.$$
Equation \eqref{energyequalityE} follows by this last equation and by the property $|\mu'|(s)=|\partial\mathcal E|(\mu(s))$ for almost all $s\in[0,+\infty[$, which is a consequence of $\mu$ being a curve of maximal slope; see \cite[Eq. 1.3.14]{AmGiSa00}.

Finally $GMM(\mu^0,\mathcal E,d_w)$ is not empty because it corresponds to the unique element of $MM(\mu^0,\mathcal E,d_w)$ whose existence is provided by the next theorem.
\end{proof}
We remark that we did not use in the proof of Theorem \ref{easyresultforE} the convexity property of Lemma \ref{lemma1convdw} in all its strength, since we applied it only to the case $\nu=\mu_0.$ In contrast the proof of the next result fully exploits Lemma \ref{lemma1convdw}.
\begin{theorem}[Evolution variational inequality for $\mathcal E$ and long time asymptotics]\label{theoremresultsforE}
Let $f,\Omega$ satisfy \eqref{fundamentalassupmtions}. Then, for any $\mu^0\in \mathcal M_+(\Omega),$ the class of minimizing movements $MM(\mu^0,\mathcal E,d_w)$ contains a unique element $t\mapsto \mu(t;\mu^0)$ which is a curve of $d_w$-maximal slope for $\mathcal E$ with respect to its \emph{strong} upper gradient $|\partial \mathcal E|.$ Moreover, the curve $\mu(\cdot;\mu^0):[0,+\infty[\rightarrow \left(\mathcal M_+(\Omega),d_w\right)$ is the unique absolutely continuous curve in $(\mathcal M(\Omega),d_w)$ such that
\begin{equation}\label{evi}\tag{EVI}
\begin{cases}
\displaystyle\frac 1 2\frac{d}{dt}d_w^2(\mu(t;\mu^0),\nu)\leq \mathcal E(\nu)-\mathcal E(\mu(t;\mu^0))&,\;\text{ for a.e. }t\in[0,+\infty[\\
\displaystyle\lim_{t\downarrow 0}d_w(\mu(t;\mu^0),\mu^0)=0& 
\end{cases},
\end{equation}
for any $\nu\in \mathcal M_+(\Omega).$

Furthermore, for any $\mu^0\in \mathcal M_+(\Omega)$,
\begin{equation}\label{asymptoticsofevi}
\lim_{t\to +\infty}d_w(\mu(t;\mu^0),\mu^*)=0,
\end{equation}
where $\mu^*$ is the unique solution of Problem \ref{problemEnricoMario} and \ref{ProblemEvGa}.
\end{theorem}
\begin{proof}
The proof of Theorem \ref{theoremresultsforE} rests upon properties \eqref{weakstartopology},  \eqref{lscdw} and \eqref{dwcompactness} and on the following stronger version of property \eqref{convexity}. Namely, $\forall \mu_0,\mu_1,\nu \in \mathscr M_c$, $t\in (0,1),\tau>0$ the following inequality holds (see also Remark \ref{Ambrosioremark}).
{\small
\begin{align}
&\mathcal E(t\mu_1+(1-t)\mu_0)+\frac{d_w(\nu,t\mu_1+(1-t)\mu_0)^2}{2\tau}\label{strongerconvexity}\\
&\leq(1-t) \left(\mathcal E(\mu_0)+\frac{d_w(\nu,\mu_0)^2}{2\tau}\right) +t\left( \mathcal E(\mu_1)+\frac{d_w(\nu,\mu_1)^2}{2\tau}\right)-t(1-t) \frac{d_w(\mu_0,\mu_1)^2}{2\tau}\notag
\end{align}}
We can apply \cite[Th. 4.0.4, Cor. 4.0.6]{AmGiSa00} (see Theorem \ref{Ambrosiotheorem2}) to complete the proof.
\end{proof}

\section{Variational approximation of $\mathcal E$}\label{SecApproxE}
Though the minimization technique provided by Theorem \ref{theoremresultsforE} is rather satisfactory in terms existence, uniqueness, and of time regularity of solutions to \eqref{evi}, it also has some disadvantages. For instance, if $\mu$ is an absolutely continuous measure having a $L^\infty$ density bounded from below by a positive constant, i.e., $\mu\geq _{\text{a.e.}}c>0,$ then one can re-write the upper envelope defining $\mathcal E(\mu)$ as 
$$2\int_\Omega fu_\mu dx- \int_\Omega |\nabla u_\mu|^2d\mu+\int_\Omega d\mu,$$
where $u_\mu$ is the \emph{unique} $W^{1,2}(\Omega)$ solution of the elliptic PDE
$$\begin{cases}
-\divergence(\mu\nabla u_\mu)=f,&\text{ in }\Omega\\
\nabla  u_\mu\cdot n=0,&\text{ on }\partial\Omega\\
\int_\Omega u_\mu dx=0&
\end{cases}.$$ 
In contrast, this is not possible in the wider generality of $\mu\in \mathcal M_+(\Omega).$ A solution $u_\mu$ of the PDE above may be defined, working in the $\mu$-dependent Sobolev space $W^{1,2}(\Omega, d\mu)$ as done in \cite{BoBuSe97}, still $u_\mu$ may be not uniquely determined. As a consequence, the convex subdifferential of $\mathcal E(\mu)$ is not in general a singleton.   

These difficulties justify the approach of this section. Namely, we approximate the functional $\mathcal E$ introducing a two parameter family of energy functionals $\{\mathcal E_{\lambda,\delta}\}_{\lambda,\delta>0}$, and we show (see Theorem \ref{convergenceminimizerstheorem}) that the minimizers of $\mathcal E_{\lambda,\delta}$ converge to the minimizer of $\mathcal E$ as we let  first $\delta\to 0^+$ and then $\lambda\to 0^+.$ 

The parameter $\lambda$ is introduced in order to cure the lack of coercivity in the definition of $\mathcal E$ that arises when $\support \mu\subset\subset \Omega$, while $\delta$ may be interpreted as a Tikonov regularization parameter that forces the minimizer of $\mathcal E_{\lambda,\delta}$ to be a Sobolev function and in particular a bounded function  for any positive $\delta$.  The advantage of this technique is that it allows us to play in better function spaces and with stronger notions of convergence. Moreover, there exists a unique $u_{\lambda,\delta}(\mu)\in W^{1,2}(\Omega)$ realizing the $\sup$ that appears in the definition of $\mathcal E_{\lambda,\delta}(\mu)$ and $u_{\lambda,\delta}(\mu)$ is uniquely determined by the elliptic PDE
$$\begin{cases}
-\divergence((\mu+\lambda)\nabla u_{\lambda,\delta})=f,&\text{ in }\Omega\\
\nabla  u_{\lambda,\delta}\cdot n=0,&\text{ on }\partial\Omega\\
\int_\Omega u_{\lambda,\delta} dx=0&
\end{cases}.$$ 
Furthermore, the couple $(\mu_{\lambda,\delta}^*,u_{\lambda,\delta}(\mu_{\lambda,\delta}^*))$, where $\mu_{\lambda,\delta}^*$ is the \emph{unique} minimizer of $\mathcal E_{\lambda,\delta},$ can be completely characterized as the solution of a PDE system, see Proposition \ref{propPDE}.

It is worth saying that if we had a more complete regularity theory for the transport density $\mu^*$ (see \cite{SaDw17} for various counterexamples) our approach would probably become much simpler, since only one of the two regularizing parameters would suffice. 

In the rest of the paper we will consider the following set of assumptions. 
\begin{assumptions}
We still assume
\begin{align}
&f=f^+-f^-\in L^\infty(\R^n),\;\;\int_{\R^n} f(x) dx=0,\notag\\
&S_f:=\support f\text{ is compact},\tag{H1}\\
&\Omega\text{ is a convex bounded domain s.t. }\Omega\supset \conv S_f.\notag
\end{align}
together with
\begin{equation}\label{passumption}
n<p<+\infty\;,\;\;\;q:=p/(p-1)\in(1,n/(n-1)).\tag{H2}
\end{equation}
\end{assumptions}
The integrability exponent $p$ is chosen in order to have the compact embedding of $W^{1,p}_0(\Omega)$ in $L^\infty(\Omega),$ indeed any function lying in $W^{1,p}_0(\Omega)$ is equivalent to a $((p-n)/p)$-H\"older continuous function. 

We define the following function spaces.
\begin{align*}
\mathscr M_0&:=\left\{\mu\in W_0^{1,p}(\Omega): \mu(x)\geq 0\;\forall x\in \Omega\right\},\\
\mathscr M_+&:=\left\{\mu\in W_0^{1,p}(\Omega): \mu(x)> 0\;\forall x\in \Omega\right\},\\
\mathscr U&:=\left\{u\in H^1(\Omega): \int_\Omega u dx=0\right\}.
\end{align*}

Let us introduce the following functionals acting on $\mathcal M_+(\Omega)$ for any $\lambda,\delta\geq 0,$ where we denote by $\mu+\lambda$ the measure $\mu+\lambda\chi_\Omega dx$,
\begin{align}
\mathcal L_\lambda(\mu,u)&:=\begin{cases} 2 \int_\Omega f u\,dx-\int_\Omega|\nabla u|^2d(\mu+\lambda) & \text{ if } u\in \mathscr C^1(\overline{\Omega}),\;\int_\Omega u dx=0\\
+\infty& \text{ otherwise}
\end{cases},\label{Lldef}\\
\mathcal L_\lambda(\mu)&:=\sup_{u\in \mathscr U} \mathcal L_\lambda(\mu,u),\\
\mathcal E_\lambda(\mu)&:=\mathcal L_\lambda(\mu)\;+\;\int_\Omega d\mu,\\
\mathcal E_{\lambda,\delta}(\mu)&:=\begin{cases}
\mathcal E_\lambda(\mu)+\delta\|\nabla \mu\|_p^p& \text{ if }\mu\in \mathscr M_0\\
+\infty& \text{ otherwise}
\end{cases},\label{Llddef}\\
\mathcal F_\lambda(\mu)&:=\scm \mathcal E_{\lambda,0}(\mu),\label{scmdef}
\end{align}
where $\scm $ stands for the lower semicontinuous envelope  with respect to the weak$^*$ topology of measures, i.e.,
$$\scm \mathcal E_{\lambda,0}(\mu):=\sup\left\{\mathcal F(\mu),\;\mathcal F\leq \mathcal E_{\lambda,0},\;\mathcal F\text{ is l.s.c. in the weak$^*$ topology}\right\},$$
and we use the convention that $\|\nabla \mu\|_p=+\infty$ if $\mu\notin W^{1,p}_0(\Omega).$

\begin{remark}
From now on we will denote by $\mu$ both a Borel measure and its density with respect of the Lebesgue measure, if $\mu$ is assumed to be absolutely continuous, as, e.g., in equation \eqref{Llddef}. This  abuse of notation simplifies our equations and it should not be of concern for the reader, due to the regularizing effect of the functionals $\mathcal E_{\lambda,\delta}.$   
\end{remark}
Before studying the iterated $\Gamma$-limit of $\mathcal E_{\lambda,\delta}$ as $\delta\to 0^+$ and $\lambda\to 0^+$, it is worth pointing out some of the properties of $\mathcal E_{\lambda,p},\mathscr M_0, \mathscr M_+,$ and $\mathscr U.$
\begin{proposition}\label{someproperties}
Under the above hypothesis (\eqref{fundamentalassupmtions}, \eqref{passumption}) the following holds.
\begin{enumerate}[i)]
\item $\mathcal E_{\lambda,\delta}$ is l.s.c., strictly convex, and densely defined on $\{\mu\in L^2(\Omega): \mu\geq 0\text{ a.e.}\}$ (endowed by the strong topology).
\item If $\mu\in \mathscr  M_0$ we have
\begin{equation}\label{pdeU}\mathcal E_{\lambda,\delta}(\mu)= \int(\mu+\lambda)|\nabla u_\mu|^2dx+\int \mu dx+\delta\|\nabla\mu\|_p^p,
\end{equation}
where $u_\mu\in \mathscr U$ is uniquely determined as the weak solution of 
$$\begin{cases}
-\divergence((\mu+\lambda)\nabla u)=f&\text{ in }\Omega\\
\partial_n u=0&\text{ on }\partial\Omega\\
\int_\Omega u dx=0
\end{cases}.$$
\item $\mathcal E_{\lambda,\delta}$  is the restriction to $\mathscr M_0$ of the functional 
\begin{equation*}
E_{\lambda,\delta}(\mu):=\begin{cases}\left(\sup_{u\in \mathscr U}2\int_\Omega fu dx-\int(\mu+\lambda)|\nabla u|^2dx\right)+\int \mu dx+\delta\|\nabla\mu\|_p^p,& \forall \mu\in W^{1,p}_0(\Omega),\;\mu>-\lambda\\
+\infty& \text{ otherwise}
\end{cases} 
\end{equation*}
\item $E_{\lambda,\delta}$ is Frechet differentiable on $\{\mu\in W^{1,p}_0: \mu>-\lambda\}$ (in the strong $W^{1,p}$-topology). Moreover at any such $\mu$ and for any $h\in W^{1,p}_0(\Omega)$ the function $F(\epsilon):=E_{\lambda,\delta}(\mu+\epsilon h)$ is real analytic in a neighborhood of $0$ and we have
\begin{align*}
\frac d {d\epsilon}F(0)=&\int (1-|\nabla u_\mu|^2)hdx+\delta p \int |\nabla \mu|^{p-2}\nabla \mu\cdot \nabla h dx\\
=&\langle 1-|\nabla u_\mu|^2-\delta p\Delta_p\mu ;h\rangle,
\end{align*}
i.e., 
\begin{equation}\label{EldGradient}
\nabla E_{\lambda,\delta}(\mu)=1-|\nabla u_\mu|^2-\delta p\Delta_p\mu,\;\;\;\;\forall \mu\in W^{1,p}_0(\Omega),\;\mu>-\lambda.
\end{equation}
\end{enumerate}
\end{proposition}
\begin{proof}$ $

\emph{i)} The functional $\mathcal E_{\lambda,\delta}$ is the sum of the supremum of a family of linear continuous functionals and of the term $\mu\mapsto \delta\|\nabla\mu\|_p^p.$ Thus the first term is lower semicontinuous and we need to prove only the lower semicontinuity of the second one. 

We notice that
$$\inf_{{\mu_k}\rightharpoonup^*\mu}\liminf_k\|\nabla \mu_k\|_p^p\geq \inf_{\begin{array}{c}{\mu_k}\rightharpoonup^*\mu,\\\|\nabla \mu_k\|_p\text{ bounded}\end{array}}\liminf_k\|\nabla \mu_k\|_p^p.$$
Therefore, when proving lower semicontinuity of $\mathcal E_{\lambda,\delta}$, we can restrict our attention to bounded sequences in $W^{1,p}_0(\Omega)$ converging to $\mu$ in the weak$^*$ topology of $\mathcal M_+(\Omega).$ Let us pick any such sequence and extract a optimizing subsequence for $\|\nabla \mu_k\|_p$. We relabel such a sequence and use the same index to simplify the notation. Being bounded in $W_0^{1,p}(\Omega)$ the sequence $\{\mu_k\}$ admits a subsequence $\{\mu_{k_j}\}_{j\in \N}$ that converges to $\tilde \mu$ weakly in $W^{1,p}_0(\Omega).$ Being the weak$^*$ topology of $\mathcal M_+(\Omega)$ an Hausdorff topology, we can get $\tilde \mu=\mu$ easily. The lower semicontinuity under weak $W^{1,p}_0(\Omega)$ limits of $\|\nabla \mu\|_p$ is standard, so we can conclude that  
$$\inf_{{\mu_k}\rightharpoonup^*\mu}\liminf_k\|\nabla \mu_k\|_p^p\geq \|\nabla \mu\|_p^p.$$ 
Let us pick $\mu\in \mathcal M_+(\Omega)$ and assume that $\mu$ does not admit a $W^{1,p}_0(\Omega)$ density with respect to the $n$-dimensional Lebesgue measure restricted to $\Omega.$ The same reasoning above shows that $\mu$ cannot be approximated in the weak$^*$ sense by any bounded sequence in $W^{1,p}_0(\Omega).$ Therefore we have 
$$\inf_{{\mu_k}\rightharpoonup^*\mu}\liminf_k\|\nabla \mu_k\|_p^p=+\infty=\|\nabla \mu\|_p^p,$$
the weak$^*$ lower semicontinuity of $\mu\mapsto \|\nabla \mu\|_p^p$ is proven. 

In order to prove convexity, without loss of generality we can restrict our attention to the domain of the functional.	Let us recall that the norm of a reflexive Banach space is uniformly convex. It follows by \cite{BuIuRe00} that the $r$-power of the norm is a totally convex (and in particular strictly convex) functional for any $r\in(1,+\infty).$ Notice also that $\mathcal E_{\lambda,\delta}(\mu)-\delta\|\nabla\mu\|_p^p$ is a convex functional (being the supremum of a family of linear functionals) and thus $\mathcal E_{\lambda,\delta}$ is totally convex and in particular strictly convex. As a consequence its minimizer on the convex set $\mathscr M_0$ is unique.

\emph{ii)} When $\mu\in \mathscr M_0$ there exists a positive finite constant $M$ such that $0<\lambda\leq \mu+\lambda\leq M$ and thus the definition of $\mathcal L_{\lambda}(\mu)$ is the variational formulation of the coercive linear elliptic problem
$$\begin{cases}
-\divergence((\mu+\lambda)\nabla u)=f&\text{ in }\Omega\\
\partial_n u=0&\text{ on }\partial\Omega\\
\int u dx=0
\end{cases},$$
that is characterized by a unique solution $u_\mu\in \mathscr U.$ Equation \eqref{pdeU} is obtained by substitution. 

\emph{iii)} Follows directly by the definition.

\emph{iv)} The Gateaux differentiability is obtained by direct computation. Note that if $\mu,h\in W^{1,p}_0$ and $\mu>-\lambda$, then $\mu_\epsilon:=\mu+\epsilon h\in W^{1,p}_0$, $\mu_\epsilon>-\lambda$ for $\epsilon$ small enough, and $\mu_\epsilon$ converges in $W^{1,p}_0(\Omega)$ and uniformly to $\mu.$ Denote by $u,u_\epsilon$ the solution of the equations
$$\begin{cases}
-\divergence((\mu+\lambda)\nabla u)=f&\text{ in }\Omega\\
\partial_n u=0&\text{ on }\partial\Omega\\
\int u dx=0
\end{cases},\;\;\;
\begin{cases}
-\divergence((\mu+\epsilon h\lambda)\nabla u_\epsilon)=f&\text{ in }\Omega\\
\partial_n u_\epsilon=0&\text{ on }\partial\Omega\\
\int u_\epsilon dx=0
\end{cases}.$$
Then we have
\begin{align*}
&\frac d{d\epsilon}F(0)=\lim_{\epsilon\to 0}\frac{\int f (u_\epsilon-u) dx+\epsilon\int hdx}{\epsilon}+\delta\frac{\|\nabla\mu_\epsilon\|_p^p-\|\nabla\mu\|_p^p}{\epsilon}\\
=&\lim_{\epsilon\to 0}\frac{\int [(\mu+\lambda)\nabla u_\epsilon\nabla u-(\mu+\epsilon h+\lambda)\nabla u_\epsilon\nabla u]dx}{\epsilon}+\int hdx+p\delta \int|\nabla \mu|^{p-2}\nabla \mu\nabla h dx\\
=&\int h(1-|\nabla u|^2) dx+p\delta \int|\nabla \mu|^{p-2}\nabla \mu\nabla h dx.
\end{align*}
Here we used that $\nabla u_\epsilon\to \nabla u$ weakly in $L^2(\Omega)$ if $\mu_\epsilon+\lambda$ are equi-bounded, uniformly positive, and converging in $L^p(\Omega)$ to $\mu+\lambda,$ see for instance \cite{Da93}. Higher order directional derivatives may be obtained by an iterative formula, see for instance \cite{CoDe15} where the real analyticity is proved.

The weak convergence of $\nabla u_\epsilon$ is too weak to show the Frechet differentiabililty of $\mathcal E_{\lambda,\delta}.$ Notice that if $\mu_\epsilon\to \mu$ in $W^{1,p}_0(\Omega)$, then (possibly passing to  equivalent representatives) $\mu_\epsilon\to \mu$ in $\mathcal C^{0,\alpha}(\overline\Omega)$ for any $\alpha\in (0,1-n/p).$ Then, using \cite[Lemma 2.5]{FaCaPu18} we can show that $u_\epsilon\to u$ in $\mathcal C^{1,\alpha}(\overline \Omega).$ This in particular implies that $(1-|\nabla u_\epsilon|^2)\to (1-|\nabla u|^2)$ in $L^q(\Omega).$ Since the Sobolev norm $\mu\mapsto\|\nabla \mu\|_p$ is well known to be Frechet differentiable (recall that here $2\leq n<p<\infty$) we have $\lim_{\nu\to \mu} |\nabla \nu|^{p-2}\nabla \nu=|\nabla \mu|^{p-2}\nabla \mu$ in $W^{-1,q}(\Omega)$ and thus the Gateaux differential of $\mathcal E_{\lambda,\delta}$ (i.e., the function $1-|\nabla u_\mu|^2-p\delta\Delta_p\mu $)  is a continuous function from $\{\mu\in W^{1,p}_0(\Omega):\mu\>-\lambda\}$ to $W^{-1,q}(\Omega)$. That is, $\mathcal E_{\lambda,\delta}$ is Frechet differentiable.  
  
\end{proof}

The following $\Gamma$-convergence result justifies the rest of our study.
\begin{theorem}[Convergence of minima and minimizers]\label{convergenceminimizerstheorem}Under the Set of Assumptions 2, the following holds.
\begin{enumerate}[(i)]
\item For any $\lambda\geq 0$, the family $\{\mathcal E_{\lambda,\delta}\}_{\delta>0}$ is decreasing, as $\delta\downarrow 0.$ The family of convex l.s.c. functionals $\{\mathcal F_\lambda\}_{\lambda>0}$ is increasing as $\lambda\downarrow 0.$
\item \emph{[$\Gamma$-convergence as $\delta\downarrow 0$]} For any $\lambda\geq 0$
$$\gammalim_{\delta\downarrow 0}\mathcal E_{\lambda,\delta}=\mathcal F_\lambda,$$
with respect to the weak$^*$ topology of $\mathcal M_+(\Omega).$
\item For any sequence $\{\lambda_i\}\downarrow 0$ and $\{\delta_j\}\downarrow 0$ let $\mu_{i,j}^*:= \argmin \mathcal E_{\lambda_i,\delta_j}.$ Then, for any fixed $i\in \N$ we can extract a subsequence $k\mapsto \mu_{i,j_k}^*$ that converges in the weak$^*$ topology of $\mathcal M(\Omega)$ to some $\mu_i^*\in \argmin_{\mathcal M(\Omega)} \mathcal F_{\lambda_i}.$
Any such subsequence satisfies
$$\lim_{k}\mathcal E_{\lambda_i,\delta_{j_k}}(\mu_{i,j_k}^*)=\min_{\mu\in \mathcal M_+(\Omega)}\mathcal F_{\lambda_i}(\mu).$$
In particular we have
\begin{equation}
\limsup_{\delta\downarrow 0}d_w(\mu_{\lambda,\delta}^*,\argmin \mathcal F_\lambda)=0\;\;\forall \lambda>0.
\end{equation}
\item \emph{[$\Gamma$-convergence as $\lambda\downarrow 0$]} For any $\mu\in \mathcal M(\Omega)$ we have
$$\gammalim_{\lambda\downarrow 0^+}\mathcal F_\lambda(\mu)=\lim_{\lambda\downarrow 0}\mathcal F_\lambda(\mu)=\sup_{\lambda>0}\mathcal F_\lambda(\mu)=\mathcal F_0(\mu),$$
with respect to the weak$^*$ topology of $\mathcal M_+(\Omega).$
Similarly, we have
$$\gammalim_{\lambda\downarrow 0^+}\mathcal E_\lambda(\mu)=\lim_{\lambda\downarrow 0}\mathcal E_\lambda(\mu)=\sup_{\lambda>0}\mathcal E_\lambda(\mu)=\mathcal E(\mu).$$
\item Let $\mu_i^*\in \argmin_{\mathcal M(\Omega)} \mathcal F_{\lambda_i}.$ Then there exists a subsequence $\mu_{i_l}^*$ converging to $\bar \mu\in \argmin_{\mathcal M(\Omega)} \mathcal F_0$ with respect to the weak$^*$ topology of $\mathcal M(\Omega).$
\item Additionally, $\bar\mu$ is the optimal transport density, i.e., 
$$\bar\mu=\mu^*,$$
Indeed the whole sequence $\{\mu_i^*\}$ satisfies
$$\mu_i^*\rightharpoonup^*\mu^*$$
and we have
$$\lim_{\lambda\downarrow 0}\lim_{\delta\downarrow 0}d_w(\mu_{\lambda,\delta}^*,\mu^*)=0.$$ 
\end{enumerate}
\end{theorem}
Before proving Theorem \ref{convergenceminimizerstheorem} we need to introduce a notion of convergence adapted to the structure of $\mathcal E_{\lambda,\delta}$ and related properties. The main ambient space we will work in is
$$L^\infty_+(\Omega):=\{\mu\in L^\infty(\Omega):\mu\geq 0\text{ a.e. in }\Omega\}.$$
\begin{definition}[$\sigma$-convergence]\label{sigmadef}
Let $\mu,\mu_j\in L^\infty_+(\Omega)$ for any $j\in \N.$ The sequence $\{\mu_j\}$ $\sigma$-converges to $\mu$  if the following conditions hold
\begin{enumerate}[i)]
\item $\sup_j\| \mu_j\|_{L^\infty(\Omega)}<+\infty$ 
\item $\mu_j(x)\to \mu(x)$ for a.e. $x\in \Omega,$
\end{enumerate}
In such a case we will write $\mu_j\sconverge{} \mu.$
\end{definition}
\begin{proposition}[$\sigma$-continuity of $\mathcal E_\lambda$]\label{propsigmacontinuity}
Let $\{\mu_j\}$ be a sequence in $L^\infty_+(\Omega)$ $\sigma$-converging to $\mu\in L^\infty_+(\Omega).$ Then, for any $\lambda>0$, we have
\begin{equation}\label{sigmacontinuity}
\mathcal E_\lambda(\mu)=\lim_j \mathcal E_\lambda(\mu_j).
\end{equation}
\end{proposition}
\begin{proof}
Let $\lambda>0$ be fixed. By standard theory of elliptic PDEs, for any $j\in \N$, there exist $u_{\lambda,\mu_j}$ and $u_{\lambda,\mu}$ that are the unique weak solution of the equation
$$\begin{cases}
-\divergence((\nu+\lambda)\nabla u) = f& \text{ in }\Omega\\
\partial_n u=0& \text{ on }\partial \Omega\\
\int_\Omega u \,dx=0
\end{cases}$$
for $\nu=\mu_j$ and $\nu=\mu$, respectively.
By the definition of weak solution it follows that
\begin{align*}
&\mathcal E_\lambda(\mu_j)=2\int_\Omega fu_{\lambda,\mu_j}dx-\int_\Omega (\mu_j+\lambda)|\nabla u_{\lambda,\mu_j}|^2 dx+\int_\Omega\mu_jdx\\
&\mathcal E_\lambda(\mu)=2\int_\Omega fu_{\lambda,\mu}dx-\int_\Omega (\mu+\lambda)|\nabla u_{\lambda,\mu}|^2 dx+\int_\Omega\mu dx
\end{align*}
By Proposition \ref{sigmaimpliesG} $u_{\lambda,\mu_j}$ converges to $u_{\lambda,\mu}$ weakly in $W^{1,2}(\Omega)$ ad thus we have
\begin{equation}
\lim_j 2\int_\Omega fu_{\lambda,\mu_j}dx=2\int_\Omega fu_{\lambda,\mu}dx.
\end{equation}
By Proposition \ref{propconvergenceofenergies} we can write
\begin{equation}
\lim_j \int_\Omega (\mu_j+\lambda)|\nabla u_{\lambda,\mu_j}|^2 dx=\int_\Omega (\mu+\lambda)|\nabla u_{\lambda,\mu}|^2 dx.
\end{equation}
By the Dominated Convergence Theorem we obtain
$$\lim_j\int_\Omega\mu_jdx=\int_\Omega \mu dx.$$
Therefore \eqref{sigmacontinuity} follows.
\end{proof}

\begin{proof}[Proof of Theorem \ref{convergenceminimizerstheorem}]$ $\\
\emph{(i)} 
Let $\mu\in \mathscr M_0$, $\lambda>0$, and $\delta_0>\delta_1>0>0$. Then
$$\mathcal E_{\lambda,\delta_0}(\mu)=\mathcal E_{\lambda,\delta_1}(\mu)+(\delta_0-\delta_1)\|\nabla \mu\|_p^p\geq  \mathcal E_{\lambda,\delta_1}(\mu).$$
Pick $\mu\in \mathscr M_0$ and $\lambda_0>\lambda_1>0>0.$ For any $\epsilon>0$ we can find $u^\epsilon_\mu\in \mathscr C^1(\overline \Omega)$ such that
\begin{align*}
&\mathcal E_{\lambda_0,0}(\mu)<2\int fu_\mu^\epsilon dx-\int (\mu+\lambda_0)|\nabla u_\mu^\epsilon|^2 dx+\int \mu dx+\epsilon\\
=&2\int fu_\mu^\epsilon dx-\int (\mu+\lambda_1)|\nabla u_\mu^\epsilon|^2 dx+\int \mu dx+\epsilon-(\lambda_0-\lambda_1)\int|\nabla u_\mu^\epsilon|^2 dx\leq\mathcal E_{\lambda_1,0}(\mu)+\epsilon.
\end{align*}
Letting $\epsilon\to 0^+$ we obtain $\mathcal E_{\lambda_0,0}(\mu)\leq \mathcal E_{\lambda_1,0}(\mu)$ for any $\lambda_0>\lambda_1.$ Therefore
$$\mathcal F_{\lambda_0}(\mu)=\sup\left\{\mathcal F(\mu),\;\mathcal F\leq \mathcal E_{\lambda_0,0},\;\mathcal F\text{ is l.s.c}\right\}\leq \sup\left\{\mathcal F(\mu),\;\mathcal F\leq \mathcal E_{\lambda_1,0},\;\mathcal F\text{ is l.s.c}\right\}=\mathcal F_{\lambda_1}(\mu). $$
\emph{(ii)} The statement follows by the monotonicity and the above mentioned lower semicontinuity by applying \cite[Prop. 5.7]{Da93} and noticing that $\mathcal E_{\lambda,0}$ is the point-wise limit of $\mathcal E_{\lambda,\delta}.$

\emph{(iii)} Let $\{\lambda_i\}\downarrow 0$ and $\{\delta_j\}\downarrow 0$ be given. Since  the functional $\mathcal E_{\lambda_i,\delta_j}$ is strictly convex as shown above, it admits a unique minimizer $\mu_{i,j}^*.$ Let us notice that, for any $i,j\in \N$, and any $\mu_j\in \mathcal M_+(\Omega)$ we have
\begin{equation}\label{trivialmassstimate}
\int_\Omega \mu_{i,j}^*dx\leq \mathcal E_{\lambda_i,\delta_j}(\mu_{i,j}^*)\leq \mathcal E_{\lambda_i,\delta_j}(\mu_j).
\end{equation}
We now assume for simplicity $\Omega=B(0,1)$ is a ball of radius $1$ centered at $0$ The general case can be treated similarly, albeit with more technicalities. Let us set ,for any $h>0$,
$$\mu^h(x):=\begin{cases}
1& \text{ if }|x|< 1-h\\
\frac{1-|x|}{h}&\text{if }1-h\leq |x|\leq 1
\end{cases}\in W^{1,p}_0(\Omega).$$
Then 
$$\nabla \mu^h(x):=\begin{cases}
0& \text{ if }|x|< 1-h\\
-\frac x{h|x|}&\text{if }1-h\leq |x|\leq 1
\end{cases}\in L^p(\Omega).$$
Clearly we have
\begin{align*}
&\|\nabla \mu^h\|_p^p\frac{1}{h^p}(|B(0,1)|-|B(0,1-h)|)=\omega_n\frac{1-(1-h)^n}{h^p}\\
=&\omega_n\sum_{s=1}^n{n\choose s}(-1)^{s+1}h^{s-p}= \mathcal O(h^{1-p})\text{  as  }h\to 0^+.
\end{align*}
Here we denoted by $\omega_n$ the standard volume of the $n$-dimensional unit ball.

Setting
$$h_j:=\delta_j^{1/(p-1)}\;,\;\;\mu_j:=\mu^{h_j},$$
we have
$$\delta_j\|\nabla \mu_j\|_p^p=\mathcal O(1)\text{  as  }j\to +\infty.$$
Hence we can pick $M\in \R $ such that $\delta_j\|\nabla \mu_j\|_p^p<M$ for any $j\in \N.$

Let $u_j$ be the weak solution of 
$$\begin{cases}
-\divergence{((\mu_j+\lambda_i)\nabla u_j)}=f& \text{ in }\Omega\\
\partial_n u_j=0&\text{ on }\partial \Omega\\
\int_\Omega u_j dx=0 
\end{cases}.$$
By the standard elliptic estimate
$$\|\nabla u_j\|_2\leq \frac{C\|f\|_2}{\min_\Omega(\mu_j+\lambda)},$$
where $C$ denotes the Poincar\'e constant of $\Omega$,
we get 
\begin{align*}
&\mathcal E_{\lambda_i,\delta_j}(\mu_j)\\
=&\int_\Omega (\mu_j+\lambda_i)|\nabla u_j|^2dx+\int_\Omega \mu_j dx+\delta_j\|\nabla \mu_j\|_p^p\\
\leq& (1+\lambda_i)\|\nabla u_j\|_2^2+|\Omega|+M\leq|\Omega|\left(1+C\|f\|_\infty^2\frac{1+\lambda_i}{\lambda_i^2}\right)+M=:M_i<\infty.
\end{align*}
We can use \eqref{trivialmassstimate} to get
$$\sup_j \int_\Omega \mu_{\lambda_i,\delta_j}^*dx\leq M_i.$$
By the compactness of the weak$^*$ topology of measures, the sequence $\{\mu_{\lambda_i,\delta_j}^*\}_{j\in \N}$ admits at least a converging subsequence and, since $\gammalim_{\delta\downarrow 0}\mathcal E_{\lambda,\delta}= \mathcal F_\lambda$, the limit point is a minimizer of $\mathcal F_\lambda$  \cite[Cor. 7.20]{Da93}.

\emph{(iv)} The two $\Gamma$ convergence results follow directly by the monotonicity and the lower semicontinuity, see \cite[Rem. 5.5]{Da93}. Indeed the $\Gamma$-limit of a decreasing family of lower semicontinuous functionals is the point-wise limit. We need to show that  for any $\mu\in \mathcal M_+(\Omega)$ we have
$$\lim_{\lambda\downarrow 0}\mathcal E_\lambda(\mu)=\mathcal E(\mu).$$
Since $\mathcal E$ is lower semicontinuous we have
$$\mathcal E(\mu)\leq \liminf_{\lambda\to 0^+}\mathcal E(\mu+\lambda)\leq \liminf_{\lambda\to 0^+}\mathcal E_\lambda(\mu)-\lambda|\Omega|=\liminf_{\lambda\to 0^+}\mathcal E_\lambda(\mu).$$
On the other hand, using the fact that $\mathcal E_\mu$ is defined as the supremum among $u\in \mathscr U$ of linear functionals, for any $\epsilon>0$ we can find $u_\lambda^\epsilon\in \mathscr U$ such that
$$\mathcal E(\mu)\geq 2\int fu_\lambda^\epsilon dx-\int |\nabla u_\lambda^\epsilon|^2 d\mu+\int d\mu\geq \mathcal E_\lambda(\mu)+\lambda\int|\nabla u_\lambda^\epsilon|^2 dx-\epsilon\geq \mathcal E_\lambda(\mu)-\epsilon.$$
Therefore we have
$$\limsup_{\lambda\to 0^+}\mathcal E_\lambda(\mu)\leq \mathcal E(\mu)\leq \liminf_{\lambda\to 0^+}\mathcal E_\lambda(\mu)$$
and equality must hold.

\emph{(v)} Let $\lambda>0.$ We notice that, by the above definitions and by the continuity of $\mu\mapsto \int _\Omega d\mu$ with respect to the weak$^*$ topology of measures, we have
\begin{align*}
\mathcal F_{\lambda}(\mu):=&\sup\{F(\mu),\;F\text{ is l.s.c. and }F(\nu)\leq \mathcal E_{\lambda,0}(\nu),\;\forall \nu\in \mathcal M_+(\Omega)\}\\
= &\sup\left\{G(\mu),\;G\text{ is l.s.c. and }G(\nu)\leq \mathcal E_{\lambda,0}(\nu)-\int_\Omega \nu dx,\;\forall \nu\in \mathcal M_+(\Omega)\right\}+\int_\Omega \mu dx\geq \int_\Omega \mu dx.
\end{align*}
Here the inequality follows by noticing that $\mathcal E_{\lambda,0}(\nu)-\int_\Omega \nu dx\geq \mathcal E_{\lambda,0}(0)\geq 0$ for any $\nu\in \mathscr M_0$. Thus
\begin{equation}
\int_\Omega d\mu\leq \mathcal F_{\lambda}(\mu),\;\;\forall \mu\in\mathcal M_+(\Omega),\;\forall \lambda>0.\label{trivialestimate}
\end{equation}
Let $\{\lambda_i\}\downarrow 0$ as $i\to \infty$ and let $\mu_i\in \argmin\mathcal F_{\lambda_i}.$ We can easily show that the mass of $\mu_i$ is bounded from above, uniformly with respect to $i$, provided we can show that
\begin{equation}
\mathcal F_\lambda(\mu)=\mathcal E_\lambda(\mu),\;\;\forall \mu\in L^\infty(\Omega),\;\forall \lambda>0,\label{claimElFl}
\end{equation}
where $\mathcal E_\lambda$ has been defined in \eqref{Lldef}. We postpone the proof of this claim that will be provided in Lemma \ref{technicallemma} below. 
Assuming \eqref{claimElFl} and using \eqref{trivialestimate}, we have
$$\int_\Omega \mu_i\leq \mathcal F_{\lambda_i}(\mu_i)\leq \mathcal F_{\lambda_i}(\chi_\Omega dx)=\mathcal E_{\lambda_i}(\chi_\Omega dx).$$
Reasoning as in the proof of \emph{(iii)} we get 
$$\int_\Omega \mu_i\leq \mathcal E_{\lambda_i}(\chi_\Omega dx)\leq |\Omega|\left(1+C^2\|f\|_\infty^2\right),\forall i\in \N,$$
note that 
$$\min_\Omega(\chi_\Omega dx+\lambda_i)= \min_\Omega \lambda_i+1\geq 1,\;\; \forall i\in \N.$$

The rest of the statement \emph{(v)} follows by the $\Gamma$-convergence of $\mathcal F_{\lambda_i}$ to $\mathcal F_0$ and by the compactness of the weak$^*$ topology of $\mathcal M_+(\Omega).$

\emph{(vi)} Let $\bar\mu$ be any cluster point of $\{\mu_i^*\}.$ We have
\begin{equation}
\mathcal E(\bar \mu)\geq \mathcal E(\mu^*)=\lim_{\lambda_i\downarrow 0}\mathcal E_{\lambda_i}(\mu^*)= \lim_{\lambda_i\downarrow 0}\mathcal F_{\lambda_i}(\mu^*)=\mathcal F_0(\mu^*)\geq \min_{\mu\in \mathcal M_+(\Omega)}\mathcal F_0(\mu),\label{estimatefrombelow}
\end{equation}
where we used (in this order) the optimality of $\mu^*$, the point-wise convergence of $\mathcal E_\lambda$ to $\mathcal E$, the $L^\infty$-regularity of $\mu^*$ and \eqref{claimElFl}, the point-wise convergence of $\mathcal F_\lambda$ to $\mathcal F_0$.

Let us assume that
\begin{equation}
\label{claimElFl2}
\mathcal E_\lambda(\mu)\leq \mathcal F_\lambda(\mu),\;\forall \mu\in \mathcal M_+(\Omega): \mathcal F_\lambda(\mu)<+\infty,\;\forall \lambda>0.
\end{equation}
We postpone the proof of this inequality to Lemma \ref{technicallemma} below.

It follows that
\begin{equation}
\mathcal E(\bar\mu)\leq\liminf_i \mathcal E_{\lambda_i}(\mu_i^*)\leq \liminf_i \mathcal F_{\lambda_i}(\mu_i^*)= \mathcal F_{0}(\bar\mu)=\min_{\mu\in \mathcal M_+(\Omega)}\mathcal F_0(\mu).\label{estimatefromabove}
\end{equation}
Here we used, the fact that $\mathcal E_\lambda$ $\Gamma$-converges to $\mathcal E$, which implies $\liminf_i \mathcal E_{\lambda_i}(\mu_{\lambda_i}^*)\geq \mathcal E(\bar\mu)$ by \cite[Prop. 8.1]{Da93}, and \eqref{claimElFl2}, the fact that $\mathcal F_\lambda$ $\Gamma$-converges to $\mathcal F_0$ and the fact that $\mu_i$ is a minimizer of $\mathcal F_{\lambda_i}$ for any $i\in \N$.

The combination of \eqref{estimatefrombelow} and \eqref{estimatefromabove} leads to
$$\mathcal E(\bar\mu)=\mathcal E(\mu^*).$$
Due to the uniqueness of the optimizer of $\mathcal E$ (see Proposition \ref{propositionexistenceanduniqueness}), we can conclude that 
$$\bar \mu=\mu^*.$$
Since $\mu^*$ is the only cluster point of the sequence $\{\mu_{\lambda_i}\}$ (and of any sequence $\{\mu_{\tilde \lambda_i}\}$ with $\{\tilde \lambda_i\}\downarrow 0$) we can conclude that the whole sequence is in fact converging to $\mu^*.$ 

Still, in order to conclude the proof of Theorem \ref{convergenceminimizerstheorem}, we are left to prove equations \eqref{claimElFl} and \eqref{claimElFl2}, see Lemma \ref{technicallemma} below. 
\end{proof}
\begin{lemma}\label{technicallemma}
Provided that the Set of Assumptions 2 holds and $\lambda>0$, we have
\begin{align}
&\mathcal E_\lambda(\mu)\leq \mathcal F_\lambda(\mu),\;\forall \mu\in \mathcal M_+(\Omega): \mathcal F_\lambda(\mu)<+\infty,\label{technicallemmaclaim1}\\
& \mathcal E_\lambda(\mu)= \mathcal F_\lambda(\mu),\;\forall \mu\in L^\infty(\Omega).\label{technicallemmaclaim2}
\end{align}
\end{lemma}
\begin{proof}
Let $\mu\in \mathcal M_+(\Omega)$ and let us assume $\mathcal F_\lambda(\mu)<+\infty.$ Since $\mathcal M_+(\Omega)$ is first countable, the relaxed functional $\mathcal F_\lambda$ has the following equivalent characterization (see \cite[Prop. 3.6]{Da93}).
\begin{align}
&\forall \mu\in \mathcal M_+(\Omega)\;\exists \{\mu_k\}_{k\in \N}\rightharpoonup^*\mu:\;\mathcal F_\lambda(\mu)\geq \limsup_k \mathcal E_{\lambda,0}(\mu_k),\label{limsupchar}\\
&\mathcal F_\lambda(\mu)\leq \liminf_k \mathcal E_{\lambda,0}(\mu_k),\;\;\forall \{\mu_k\}_{k\in \N}\rightharpoonup^*\mu.\label{liminfchar}
\end{align}  
Let us pick $\mu_k$ as in \eqref{limsupchar}. Since we have
$$+\infty> \mathcal F_\lambda(\mu)=\limsup_k \mathcal E_{\lambda,0}(\mu_k),$$
we must have $\mu_k\in W^{1,p}_0 (\Omega)$ for $k$ large enough. Therefore, by the lower-semicontinuity of $\mathcal E_\lambda$, we can write
$$ \mathcal F_\lambda(\mu)=\limsup_k \mathcal E_{\lambda,0}(\mu_k)=\limsup_k \mathcal E_{\lambda}(\mu_k)\geq \liminf_k \mathcal E_{\lambda}(\mu_k)\geq \mathcal E_\lambda(\mu),$$
from which \eqref{technicallemmaclaim1} follows.

Let $\mu\in L^\infty(\Omega)$ and $\mu\geq 0$ a.e. in $\Omega.$ We denote by $\{\mu_k\}$ the sequence approximations to $\mu$
$$\mu_k:=(\mu\cdot \chi_{\Omega_{1/k}})\ast \eta_k,$$
where $\Omega_{1/k}$ is the set $\{x\in \Omega: d(x, \partial\Omega)>1/k\},$ and $\eta_k$ is a standard mollifier of step $1/k.$ Note that 
\begin{enumerate}[a)]
\item $\mu_k\in W^{1,p}_0(\Omega)$
\item $\mu_k(x)\to\mu(x),\;$a.e. in $\Omega,$
\item $\lambda\leq \lambda+\mu_k\leq M$ uniformly in $k.$
\end{enumerate} 
The combination of (b) and (c) implies $\mu_k\sconverge \mu$, while (a) ensures that $\mathcal E_{\lambda}(\mu_k)=\mathcal E_{\lambda,0}(\mu_k)$ for any $\lambda>0,$ $k\in \N.$ Therefore, using Proposition \ref{propsigmacontinuity}, we have
$$\mathcal E_\lambda(\mu)=\lim_k \mathcal E_\lambda(\mu_k)=\lim_k \mathcal E_{\lambda,0}(\mu_k)\geq \liminf_k \mathcal E_{\lambda,0}(\mu_k)\geq \mathcal F_\lambda(\mu),$$
where we used \eqref{liminfchar} in the last inequality. In order to conclude the proof of \eqref{technicallemmaclaim2}, we need to show the reverse inequality. We notice that $\mu\in L^\infty(\Omega)$ implies $\mathcal F_\lambda(\mu)<+\infty$ and hence \eqref{technicallemmaclaim1} holds.
\end{proof}

In order to obtain a PDE characterization of the minimizers of $\mathcal E_{\lambda,\delta}$ for $\lambda,\delta>0,$ we study the subdifferentiability of these functionals.
Let us recall the definition of the subdifferential $\partial G$ of a convex function $G$ on a Banach space X.
$$\partial G(x):=\{\xi \in X^*: G(y)-G(x)\geq \langle\xi;y-x\rangle,\;\forall y\in X\},\;\forall x\in \Dom(G).$$
Note that for a convex functional $G$ this set is precisely the Frechet subdifferential of $G$. We introduce also the set 
$$\minpartial G(x):=\{\xi \in \partial G(x): \|\xi\|_{X^*}=\min_{\eta\in \partial G(x)}\|\eta\|_{X^*}\}.$$
Notice that for a convex coercive functional $G$ and for any $\mu\in \Dom(G)$ the following are equivalent
$$\mu\in \argmin G,\;\;\;0\in \partial G(\mu),\;\;\;\;\|\minpartial G(\mu)\|=0.$$  

\begin{proposition}
The functional $\mathcal E_{\lambda,\delta}$ is subdifferentiable in $L^2(\Omega)$ at $\mu\in \mathscr M_0$ (i.e., $\partial \mathcal E_{\lambda,\delta}(\mu)\neq\emptyset$) if and only if
\begin{equation}\label{domainsubdifferential}
\mu\in \{\mu\in W^{1,p}_0:\mu(x)\geq 0,\;|\nabla \mu|^{p-2}\nabla \mu\in W^{1,2}(\{\mu>0\})\}=\Dom(\partial \mathcal E_{\lambda,\delta}).
\end{equation}
More precisely, for any $\mu\in \Dom(\partial \mathcal E_{\lambda,\delta}),$ we have
\begin{equation}\label{subdiff}
\partial \mathcal E_{\lambda,\delta}(\mu)=\begin{cases}
\nabla E_{\lambda,\delta}(\mu)=1-|\nabla u_\mu|^2-\delta p\Delta_p\mu& \text{ if }\mu\in \mathscr M_+\\
\{\xi \in L^2(\Omega):\langle \xi;h\rangle\leq\langle\nabla E_{\lambda,\delta}(\mu);h\rangle,\;\forall h\in L^2(\Omega): h\geq 0 \text{ on }\{\mu=0\}\}&\text{ otherwise} 
\end{cases}.
\end{equation}
For any $\mu\in \Dom(\partial \mathcal E_{\lambda,\delta})$ we have
\begin{align}\label{minimalsubdiff}
\minpartial \mathcal E_{\lambda,\delta}(\mu)=&\begin{cases}
1-|\nabla u_\mu|^2-\delta p\Delta_p\mu& \text{ if }\mu\in \mathscr M_+\\
\xi^*(\mu)&\text{ if }\mu\in \mathscr M_0\setminus \mathscr M_+ 
\end{cases},\text{ where }\\
\xi^*(\mu):=&(1-|\nabla u_\mu|^2-p\delta \Delta_p \mu)\chi_{\{\mu>0\}}-(1-|\nabla u_\mu|^2)^-\chi_{\{\mu=0\}}.\label{xistar}
\end{align}
\normalsize
\end{proposition}
\begin{proof}
Since $\mathcal E_{\lambda,\delta}$ is convex, the Frechet subdifferential and the convex subdifferential are the same, therefore
$$\partial \mathcal E_{\lambda,\delta}(\mu)=\{\xi\in L^2(\Omega): \mathcal E_{\lambda,\delta}(\nu)-\mathcal E_{\lambda,\delta}(\mu)\geq \langle \xi;\nu-\mu\rangle\;\forall \nu\in L^2(\Omega)\}. $$
Note that the inequality constraint above is void for any $\nu\in L^2(\Omega)\setminus \mathscr M_0$ (and in particular for any $\nu\in L^2(\Omega)\setminus W^{1,p}_0(\Omega)$), thus we can write
\begin{align*}
\partial \mathcal E_{\lambda,\delta}(\mu)=&\{\xi\in L^2(\Omega): \mathcal E_{\lambda,\delta}(\nu)-\mathcal E_{\lambda,\delta}(\mu)\geq \langle \xi;\nu-\mu\rangle\;\forall \nu\in W^{1,p}_0(\Omega)\}\\
\subseteq&\{\xi\in W^{-1,q}(\Omega): \mathcal E_{\lambda,\delta}(\nu)-\mathcal E_{\lambda,\delta}(\mu)\geq \langle \xi;\nu-\mu\rangle\;\forall \nu\in W^{1,p}_0(\Omega)\}.
\end{align*}
Let us assume $\mu\in \mathscr M_+.$ Then, using (iv) of Proposition \ref{someproperties}, we have
\begin{align*}
\partial \mathcal E_{\lambda,\delta}(\mu)\subseteq&\{\xi\in W^{-1,q}(\Omega): \mathcal E_{\lambda,\delta}(\nu)-\mathcal E_{\lambda,\delta}(\mu)\geq \langle \xi;\nu-\mu\rangle\;\forall \nu\in W^{1,p}_0(\Omega)\}\\
=&\{\nabla E_{\lambda,\delta}\}=1-|\nabla u_\mu|^2+p\delta \Delta_p \mu.
\end{align*}
Therefore $\partial \mathcal E_{\lambda,\delta}(\mu)$ is not empty at $\mu\in \mathscr M_+$ if and only if $1-|\nabla u_\mu|^2+p\delta \Delta_p \mu\in L^2(\Omega).$ Note that, by elliptic regularity, $1-|\nabla u_\mu|^2\in L^2(\Omega)$ since $\mu\in C^{0,\alpha}(\overline\Omega)$ (for any $\alpha\in[0, 1-n/p]$) implies $u_\mu\in\mathscr C^{1,\alpha}(\overline \Omega),$ see \cite[Lemma 2.5]{FaCaPu18}. Hence $\mathcal E_{\lambda,\delta}$ is subdifferentiable at $\mu\in \mathscr M_+$ if and only if $\Delta_p\mu\in L^2(\Omega).$ By the main result of \cite{CiMa19}, there exist two positive finite constants $c_1,c_2$ (depending only on $\Omega$, $n$ and $p$) such that
$$c_1\|\Delta_p\mu\|_{L^2(\Omega)}\leq \||\nabla \mu|^{p-2}\nabla \mu\|_{W^{1,2}(\Omega)}\leq c_2\|\Delta_p\mu\|_{L^2(\Omega)},\;\;\forall \mu\in W^{1,p}_0(\Omega).$$
Hence $\partial \mathcal E_{\lambda,\delta}(\mu)$ is not empty at $\mu\in \mathscr M_+$ if and only if
$$\mu\in \{\mu\in W^{1,p}_0,\;|\nabla \mu|^{p-2}\nabla \mu\in W^{1,2}(\Omega)\}.$$

Following a similar reasoning we can show that, if $\mu\in \mathscr M_0\setminus \mathscr M_+$, we have
$$\partial_{L^2}\mathcal E_{\lambda,\delta}(\mu)\subseteq \partial_{W^{1,p}_0}\mathcal E_{\lambda,\delta}(\mu)=\{\xi\in W^{-1,q}(\Omega): \langle \xi;\nu-\mu\rangle\leq \langle \nabla E_{\lambda,\delta}(\mu);\nu-\mu\rangle,\;\forall \nu\in \mathscr M_0\}.$$
We want to conclude that the first set is precisely the intersection of the latter with $L^2(\Omega).$ To this aim let us assume by contradiction that we can pick $\xi\in L^2(\Omega)$ such that, for a $\nu\in \mathscr M_0$ we have
$$ \langle \nabla E_{\lambda,\delta}(\mu);\nu-\mu\rangle<\langle \xi;\nu-\mu\rangle\leq \mathcal E_{\lambda,\delta}(\nu)-\mathcal E_{\lambda,\delta}(\mu).$$
Let $\mu_t:=(1-t)\mu+t\nu$ for any $t\in [0,1].$ By differentiability of $E_{\lambda,\delta}$ we can write
$$ \langle \nabla E_{\lambda,\delta}(\mu);t(\nu-\mu)\rangle<\langle \xi;t(\nu-\mu)\rangle\leq \langle \nabla E_{\lambda,\delta}(\mu);t(\nu-\mu)\rangle+\frac {t^2}{2}(\nu-\mu)\Hess E_{\lambda,\delta}(\mu_{s(t)})(\nu-\mu),\;\;s(t)\in[0,t].$$
Therefore we have
 $$ \langle \nabla E_{\lambda,\delta}(\mu);(\nu-\mu)\rangle<\langle \xi;(\nu-\mu)\rangle\leq \langle \nabla E_{\lambda,\delta}(\mu);(\nu-\mu)\rangle+\frac {t}{2}(\nu-\mu)\Hess E_{\lambda,\delta}(\mu_{s(t)})(\nu-\mu),\;\;s(t)\in[0,t].$$
But taking the limit as $t\to 0^+$ and using the convexity of $E_{\lambda,\delta}$ we get a contradiction, thus no such $\xi\in L^2$ can exist. Therefore we get
\begin{align}
\partial_{L^2}\mathcal E_{\lambda,\delta}(\mu)=&\{\xi\in L^2(\Omega): \langle \xi;\nu-\mu\rangle\leq \langle \nabla E_{\lambda,\delta}(\mu);\nu-\mu\rangle,\;\forall \nu\in \mathscr M_0\}\notag\\
=&\{\xi\in L^2(\Omega): \langle \xi;h\rangle\leq \langle \nabla E_{\lambda,\delta}(\mu);h\rangle,\;\forall h\in  W^{1,p}_0(\Omega), h\geq 0 \text{ on }\{\mu=0\}\}.\label{diffcharacterization}
\end{align}
In order to conclude the proof of \eqref{domainsubdifferential} we are left to show that, also for $\mu\in \mathscr M_0\setminus \mathscr M_+$ we have
\begin{equation}
\partial_{L^2}\mathcal E_{\lambda,\delta}(\mu)=\emptyset\;\;\text{ if and only if }|\nabla \mu|^{p-2}\nabla \mu\notin W^{1,2}(\{\mu>0\}).
\end{equation} 
One implication is quite evident. Indeed if $|\nabla \mu|^{p-2}\nabla \mu\in W^{1,2}(\{\mu>0\})$ then $\Delta_p \mu\in L^2(\{\mu>0\}).$ Therefore, setting 
$\xi=(1-|\nabla u_\mu|^2-p\delta \Delta_p \mu)\chi_{\{\mu>0\}}-(1-|\nabla u_\mu|^2)\chi_{\{\mu=0\}}\in L^2(\Omega)$
we can check that $\xi$ satisfies \eqref{diffcharacterization} to obtain $\xi\in \partial \mathcal E_{\lambda,\delta}(\mu),$ which therefore is not empty.

In order to prove the converse implication we pick $\mu\in \mathscr M_0$ such that $\Delta_p\mu\notin L^2(\{\mu>0\}).$ Then we can find a sequence $h_n\in \mathscr C^\infty_c(\{\mu>0\})$ with $\|h_n\|_2=1$ such that
$$\inf_{h\in \mathscr C^\infty_c(\{\mu>0\}),\|h\|_2\leq 1}\langle -\Delta_p \mu;h\rangle=\lim_n \langle -\Delta_p \mu;h_n\rangle=-\infty.$$
Let us pick $\xi \in \partial \mathcal E_{\lambda,\delta}(\mu).$ We have
\begin{align*}
&\inf\{\langle \xi;h\rangle:\; h\in \mathscr C^\infty_c(\Omega),\|h\|_2\leq 1\}\leq\inf\{\langle \xi;h\rangle:\; h\in \mathscr C^\infty_c(\Omega),\|h\|_2\leq 1,h\geq 0 \text{ on }\{\mu=0\}\}\\
\leq&\inf\{\langle \nabla E_{\lambda,\delta};h\rangle:\; h\in \mathscr C^\infty_c(\{\mu>0\}),\|h\|_2\leq 1\}\leq\|1-|\nabla u_\mu|^2\|_2+p\delta\lim_n\langle \Delta_p\mu;h_n\rangle=-\infty
\end{align*}
Therefore $\|\xi\|_2=+\infty$, so $\xi\notin \partial_{L^2}\mathcal E_{\lambda,\delta}(\mu).$ 

We are left to prove \eqref{minimalsubdiff}. When $\mu\in \mathscr M_+\cap \Dom(\partial \mathcal E_{\lambda,\delta})$ there is nothing to prove, because we already shown that the subdifferential is a singleton. So we restrict our attention to $\mu\in (\mathscr M_0\setminus \mathscr M_+)\cap \Dom(\partial \mathcal E_{\lambda,\delta}).$  For any $\xi \in \partial \mathcal E_{\lambda,\delta}(\mu)$ we can write $\xi=\xi_1+\xi_0$, where  $\xi_1=\xi\chi_{\{\mu>0\}},$ and $\xi_0=\xi\chi_{\{\mu=0\}}.$ It is not hard to see that 
$$\xi_1=\nabla E_{\lambda,\delta}(\mu)\chi_{\{\mu>0\}}.$$
Indeed using \eqref{diffcharacterization} we have
$$\langle \xi_1;h\rangle \leq \langle \nabla E_{\lambda,\delta}(\mu)\chi_{\{\mu>0\}};h\rangle,\;\forall h\in \mathscr C^\infty_c(\{\mu>0\}),$$
so we have
$$\|\xi_1-\nabla E_{\lambda,\delta}(\mu)\chi_{\{\mu>0\}}\|_2=\sup_{h\in \mathscr C^\infty_c(\{\mu>0\}),\|h\|\leq 1} \langle \xi_1-\nabla E_{\lambda,\delta}(\mu)\chi_{\{\mu>0\}};h\rangle=0.$$
On the other hand $\xi^*(\mu)\in \partial \mathcal E_{\lambda,\delta}(\mu)$ because it satisfies \eqref{diffcharacterization}.  This is easy to see since, for any $h\in L^2(\Omega)$ we can write
$$h=\lim_n h_n=\lim_n h_n^++h_n^0,$$
where $h_n^+$ and $h_n^0$ are $\mathscr C^\infty_c$ with $\support h_n^+\subset\subset \{\mu>0\}$ and $\support h_n^0\subset\subset \{\mu=0\}$
\begin{align*}
&\langle \nabla E_{\lambda,\delta}(\mu);h\rangle=\lim_n\langle \nabla E_{\lambda,\delta}(\mu);h_n\rangle=\lim_n\int_{\{\mu>0\}}(1-|\nabla u_\mu|^2-p\delta \Delta_p\mu)h_n^+dx+\int_{\{\mu=0\}}(1-|\nabla u_\mu|^2)h_n^0dx\\
=&\int_{\{\mu>0\}}(1-|\nabla u_\mu|^2-p\delta \Delta_p\mu)hdx+\int_{\{\mu=0\}}(1-|\nabla u_\mu|^2)hdx\geq\langle \xi^*;h\rangle,\;\;\forall h: h\geq 0\text{ on }\{\mu=0\}.
\end{align*}
It is worth noticing that here the boundary term of the $p$-Laplacian gives no contribution on the set $\{\mu=0\}$ due to the approximation by compactly supported functions that vanish on the boundary of $\{\mu>0\}$.

 Thus we can write $\xi=\xi^*+\phi$ for any $\xi\in \partial \mathcal E_{\lambda,\delta}(\mu)$, where, in order to get $\xi^*+\phi\in \partial \mathcal E_{\lambda,\delta}(\mu),$ we need to impose
$$\phi\leq (\nabla E_{\lambda,\delta}(\mu))^+\text{ a.e. in }\{\mu=0\}.$$

Let $S^+:=\support (\nabla E_{\lambda,\delta}(\mu))^+$, $S^-:= S_\mu^c\setminus S^+.$ Let us define $\phi^+:=\phi\chi_{S^+}\leq (\nabla E_{\lambda,\delta}(\mu))^+,$ $\phi^-:=\phi\chi_{S^-}\leq 0.$ We have
$$\|\xi^*+\phi\|_2^2=\|\xi^*\|_2^2+\|\phi^+\|_2^2+\|\phi^-\|_2^2-2\langle (\nabla E_{\lambda,\delta}(\mu)\chi_{\{\mu=0\}})^-;\phi^-\rangle\geq\|\xi^*\|_2^2.$$
Note that the inequality follows by the sign of $\phi^-$ and that the equality holds only in the case $\phi=0.$
\end{proof}

Due to the above result, we can characterizer the minimizer of $\mathcal E_{\lambda,\delta}$ by the following  regularized version of Monge Kantorovich equations.
\begin{proposition}\label{propPDE}
There exists a unique solution $(\mu_{\lambda,\delta}^*,u_{\lambda,\delta}^*)\in\mathscr M_0\times \mathscr U$ of 
\begin{equation}\label{optimumpde}
\begin{cases}
1-|\nabla u|^2-\delta p \Delta_p \mu=0& \text{ on }\support \mu\\
|\nabla u|^2\leq 1& \text{ on }\{\mu=0\}\\
-\divergence((\mu+\lambda)\nabla u)=f& \text{ in }\Omega\\
\mu\geq 0& \text{ in }\Omega\\
\mu=0& \text{ on }\partial \Omega\\
\partial_n u=0& \text{ on } \partial\Omega\\
\int_\Omega u dx=0
\end{cases}.
\end{equation} 
Moreover $\mu_{\lambda,\delta}^*=\argmin \mathcal E_{\lambda,\delta}$ and $u_{\lambda,\delta}^*$ realizes the supremum defining $\mathcal E_{\lambda,\delta}$, i.e., 
$$\mathcal E_{\lambda,\delta}(\mu_{\lambda,\delta^*})=\int(\mu_{\lambda,\delta}^*+\lambda)|\nabla u_{\lambda,\delta}^*|^2dx+\int\mu_{\lambda,\delta}^*dx+\delta\|\nabla \mu_{\lambda,\delta}^*\|_p^p .$$ 
\end{proposition}
\begin{proof}
Recall that we already shown that $\mathcal E_{\lambda,\delta}$ has a unique minimizer $\mu^*_{\lambda,\delta}.$ Thus we must have $0\in \partial \mathcal E_{\lambda,\delta}(\mu^*_{\lambda,\delta})$. This differential inclusion is equivalent to the existence of $u_{\lambda,\delta}^*\in \mathscr U$ such \eqref{optimumpde} holds. Due to the coercivity of the elliptic equation 
$$\begin{cases}
-\divergence((\mu+\lambda)\nabla u)=f& \text{ in }\Omega\\
\partial_n u=0& \text{ on } \partial\Omega\\
\int_\Omega u dx=0\end{cases},$$
the function $u_{\lambda,\delta^*}$ needs to be unique.

If conversely we assume that there exists a couple $(\mu_{\lambda,\delta}^*,u_{\lambda,\delta}^*)\in\mathscr M_0\times \mathscr U$ satisfying \eqref{optimumpde}, then we have $0\in \partial \mathcal E_{\lambda,\delta}(\mu^*_{\lambda,\delta})$ and $\mu_{\lambda,\delta}^*$ must be a minimizer of $\mathcal E_{\lambda,\delta}$ and thus the unique minimizer.
\end{proof}

\section{Dynamical minimization of $\mathcal E_{\lambda,\delta}$}\label{SecMinimizeEld}
In view of Theorem \ref{convergenceminimizerstheorem} it is worth studying the gradient flows of the functionals $\mathcal E_{\lambda,\delta}$ and their long time asymptotics. This task can be accomplished following the technique we exploited in Section \ref{SecMinimizeE} with a metric constructed as $d_w$ and inducing the weak topology of $W^{1,p}_0(\Omega).$ This strategy would lead to the same kind of results of Section \ref{SecMinimizeE}.

Instead we aim at a more neat and possibly PDE-based characterization of the flow that indeed justifies the approximation of $\mathcal E$ by the family $\mathcal E_{\lambda,\delta}.$ To this goal we study look at the gradient flow of 
$$\mathcal E_{\lambda,\delta}:L^2(\Omega)\rightarrow \R\cup \{+\infty\}$$
and we exploit both the regularizing effect of the functional and the identification of $L^2$ and its dual. This approach turns out to be profitable, as shown by the next two results. 

\begin{theorem}[Existence of $L^2$ Gradient Flow]\label{Thl2gf}
Let $\Omega,\lambda,p,f$ be as above. Let $\mu^0\in \mathscr M_0.$ Then the gradient flow equation
\begin{equation}\label{l2gf}
\begin{cases}
\mu'(t)=-\minpartial \mathcal E_{\lambda,\delta}(\mu(t)),& t>0\\
\mu(0)=\mu^0
\end{cases},
\end{equation}
that can be written as
\begin{equation}\label{evolution_pde}
\begin{cases}
\frac d{dt}\mu(t,x)=[|\nabla u(t,x)|^2-1+\delta p\Delta_p\mu(t,x)]\chi_{\{\mu>0\}}+[(|\nabla u|^2-1)\chi_{\{\mu=0\}}]^+&\text{ in }[0,+\infty[\times\Omega\\
-\divergence((\mu(t,x)+\lambda)\nabla u(t,x))=f(x)&\text{ in }[0,+\infty[\times\Omega\\
\mu(t,x)=0&\text{ in }[0,+\infty[\times\partial\Omega\\
\partial_n u(t,x)=0,\int_\Omega u(t,x) dx=0&\text{ in }[0,+\infty[\times\partial\Omega\\
\mu(0,x)=\mu^0&\text{ for any }x\in \Omega
\end{cases},
\end{equation}
has an unique absolutely continuous (and almost everywhere differentiable) solution $[0,+\infty[\ni t\mapsto\mu(t;\mu^0)\in L^2(\Omega)$ which is a curve of maximal slope for the strong upper gradient $\|\minpartial \mathcal E_{\lambda,\delta}\|_2.$ 

Moreover we have
\begin{equation}\label{Energyequality}
\mathcal E_{\lambda,\delta}(\mu(t+\tau))-\mathcal E_{\lambda,\delta}(\mu(t))=-\int_t^{t+\tau}|\mu'(s)|^2ds=-\int_t^{t+\tau}\|\minpartial \mathcal E_{\lambda,\delta}\|^2(\mu(s))ds.
\end{equation}
Then $\mu(t;\mu^0)\in \Dom(\partial E_{\lambda,\delta})$ for almost every $t>0$, $\{\mu(t;\mu^0),\;t>0\}$ is bounded in $W^{1,p}_0(\Omega)$ and $t\mapsto\mu(t;\mu^0)$ is continuous with respect to the $W^{1,p}_0$ topology.
\end{theorem}
\begin{proof}
The functional $\mathcal E_{\lambda,\delta}$ is convex, l.s.c., proper and bounded below by $0.$ The sublevel sets of $\mathcal E_{\lambda,\delta}$ are bounded in $W^{1,p}_0(\Omega)$ and thus strongly compact in $L^2$ (and weakly compact in $W^{1,p}_0(\Omega).$) We can apply \cite[Th.2.3.7]{AmGiSa00} to show the existence of the flow and equation \eqref{Energyequality}. Note that $\minpartial E_{\lambda,p}(\mu(t))$ is single valued because the $L^2$ norm is strictly convex. Since $L^2$ is isometrically isomorphic to its dual (not simply a reflexive Banach space), for any couple of solutions $t\mapsto\mu(t)$ and $t\mapsto\nu(t)$ of \eqref{l2gf} we can write
$$\frac 1 2\frac d{dt}\|\mu(t)-\nu(t)\|_2^2=\langle\mu(t)-\nu(t);\mu'(t)-\nu'(t)\rangle=-\langle\mu(t)-\nu(t);\minpartial E_{\lambda,\delta}(\mu(t))-\minpartial \mathcal E_{\lambda,\delta}(\nu(t))\rangle\leq 0.$$
Here the inequality follows by the monotonicity of the subdifferential of a convex operator. The uniqueness of the solution easily follows.

The energy equality \eqref{Energyequality} forces the function $t\mapsto \mathcal E_{\lambda,\delta}(\mu(t;\mu^0)$ to be continuous and the trajectory to be bounded in $W^{1,p}_0.$ In order to show it, let us pick any $\hat t>0$ and any sequence $0<t_j\to \hat t.$ Since $\mu_j:=\mu(t_j;\mu^0)$ is bounded in $W^{1,p}_0(\Omega)$ we can extract a weakly converging subsequence $\mu_{j_k}$. Since the starting sequence is converging in $L^2$ to $\hat \mu:=\mu(\hat t;\mu^0)$, we have $\mu_{j_k}\to \hat \mu$ weakly in $W^{1,p}_0(\Omega).$ Possibly passing to a further subsequence and relabeling it, we can assume that the convergence is indeed uniform, thanks to the compact embedding of $W^{1,p}_0(\Omega)$ in $\mathscr C^{0,\alpha}(\overline\Omega)$ for $\alpha\in [0,1-n/p].$ The uniform convergence $\mu_{j_k}\to \hat \mu$ implies 
$$a_{j_k}:=\int \mu_{j_k} dx+ \int(\mu_{j_k}+\lambda)|\nabla u_{j_k}|^2dx\to\int\hat\mu dx+ \int(\hat\mu+\lambda)|\nabla \hat u|^2 dx=:\hat a,$$
where $u_{j_k}:=u_{\mu_{j_k}},$ and $\hat u:=u_{\hat \mu}.$ Using the continuity of $\mathcal E_{\lambda,\delta}$ along the sequence, we can write
$$\lim_k\|\nabla \mu_{j_k}\|_p^p=\frac 1{p\delta}\left(\lim_k \mathcal E_{\lambda_\delta}(\mu_{j_k}) -a_{j_k}\right)=\frac 1{p\delta}\left(\mathcal E_{\lambda_\delta}(\hat\mu) -\hat a\right) =\|\nabla \hat \mu\|_p^p.$$
The combination of the weak convergence in $W^{1,p}_0$ with the convergence of norms leads to the strong $W^{1,p}_0$ convergence of $\mu_{j_k}$ to $\hat \mu$. Since the limit is not depending on the particular sequence it follows that 
$$\lim_{t\to \hat t}\|\mu(t;\mu^0)-\mu(\hat t;\mu^0)\|_{1,p}=0,\;\;\forall \hat t>0, \mu^0\in \mathscr M_0.$$
\end{proof}

\begin{theorem}[Long time behavior of the gradient flow]\label{Thl2as}
Let $\mu^0\in \mathscr M_0$ and let $[0,+\infty[\ni t\mapsto\mu(t;\mu^0)$ be the solution of \eqref{evolution_pde}. Then we have
\begin{equation}
\lim_{t\to +\infty}\|\mu(t;\mu^0)-\mu_{\lambda,\delta}^*\|_{W^{1,p}}=0,\;\;\;\forall \mu^0\in \mathscr M_0,
\end{equation}
where
$\mu_{\lambda,\delta}^*:=\argmin_{\mu\in \mathscr M_0}\mathcal E_{\lambda,\delta}(\mu).$
\end{theorem}
\begin{proof}
By \eqref{Energyequality} it follows that for any $\mu^0\in \mathscr M_0$ the function $[0,+\infty[\ni t\to \|\minpartial \mathcal E_{\lambda,\delta}(\mu(t;\mu^0)\|_2$ is square summable. 
Therefore we can find a sequence $t_k\to +\infty$ such that  $\|\minpartial \mathcal E_{\lambda,\delta}(\mu_k)\|_2\to 0.$ Possibly passing to a subsequence and relabeling it, we can assume that $ \mu_k:=\mu(t_k;\mu^0)\to \hat \mu\;\text{ in }L^2(\Omega).$
The graph of the subdifferential of convex and lower semicontinuous functionals is lower semicontinuous in the strong$\times$weak topology, therefore 
$$\min_{ \xi \in \minpartial \mathcal E_{\lambda,\delta}(\hat\mu)}\|\xi\|_2\leq \liminf_k \|\minpartial \mathcal E_{\lambda,\delta}(\mu_k)\|_2=0.$$
Thus $0\in \partial \mathcal E_{\lambda,\delta}(\hat\mu)$, and this equation characterizes the minimizers of $\mathcal E_{\lambda,\delta}.$ By Proposition \ref{propPDE} above we can conclude that $\hat \mu=\mu_{\lambda,\delta}^*.$ Note that, repeating the argument of the last part of the proof of Theorem \ref{Thl2gf}, we can prove that the convergence $\mu(t;\mu^0)\to \mu_{\lambda,\delta}^*$ holds indeed in the strong topology of $W^{1,p}_0(\Omega).$
\end{proof}

\appendix
\section{Some tools from metric analysis}\label{app1}
\subsection{Minimizing Movements}\label{subsectionMM}
We recall here some basic definitions and facts from metric analysis and the theory of gradient flows. We refer the reader to \cite{AmGiSa00} for  an extensive treatment of the subject.

Given a complete metric space $(\mathscr S,d)$ and a lower semicontinuous functional $\phi\not\equiv+\infty$ and a sequence of time steps $\boldsymbol \tau:=\{\tau_k\}_{k\in \N}$, $\tau_k>0$, $\sum_{k=0}^{+\infty}\tau_k=+\infty$, for any $\mu_0\in \Dom(\phi)\subseteq \mathscr S$ (the set of points such that $\phi<+\infty$) one can find      a sequence of  minimizers implicitly defined by setting
$$\mu_{k+1}\in \argmin_{\nu\in \mathscr S}\Phi(\nu,\mu_k,\tau_k):=\argmin_{\nu\in \mathscr S}\left(\phi(\nu)+\frac{d^2(\mu_k,\nu)}{2\tau_k}\right).$$
This leads to a so called \emph{discrete trajectory}
$$\mu_{\boldsymbol \tau}(t):=\mu_{\hat k(t)},$$
where $\hat k(t)$ is the greatest integer for which $\sum_{k=1}^{\hat k(t)}\tau_k<t.$
For a given notion of convergence $\sigma$ (not necessarily a topology) in $\mathscr S$ possibly different from the one induced by $d$, and for any $\mu^0\in \Dom(\phi),$ one can look to the class of all curves $\mu:[0,+\infty[\ni t\mapsto \mathscr S$ such that, for any sequence of partitions $\{\boldsymbol \tau^n\}$ as above, such that
$$\lim_{n\to +\infty}|\boldsymbol \tau^n|:=\lim_{n\to +\infty}\sup_k \tau^n_k=0,$$
we have
\begin{align}
&\lim_n \phi(\mu_{\boldsymbol \tau^n}(0))=\phi(\mu^0)\notag\\
&\limsup_n d(\mu_{\boldsymbol \tau^n}(0),\mu^0)<+\infty\label{MMdef}\\
&\silim_n \mu_{\boldsymbol \tau^n}(t)=\mu(t),\;\;\forall t\in[0,+\infty[.\notag
\end{align}
The class of all such curves is termed the class of \emph{Minimizing Movements} for $\phi$ starting at $\mu^0$ with respect to $d$, denoted it by $MM(\mu^0,\phi,d).$

If in the definition of $MM(\mu^0,\phi,d)$ the requirement of equations \eqref{MMdef} holding true for every sequence of partitions $\{\boldsymbol \tau^n\}$ shrinking to $0$ is replaced by the requirement of the existence of just one sequence of partitions $\{\boldsymbol \tau^n\}$ shrinking to $0$ such that equations \eqref{MMdef} hold, then we obtain the definition of \emph{Generalized Minimizing Movements} for $\phi$ starting at $\mu^0$ with respect to $d$; we denote such a class by $GMM(\mu^0,\phi,d).$  We remark that  $MM(\mu^0,\phi,d)$ either contains a single element or it is the empty set, while  $GMM(\mu^0,\phi,d)$ can be empty, contain one or several curves. 
\subsection{Upper gradients, slopes and curves of maximal slope}\label{subsectionmetric}
Let $(\mathscr S,d)$ denote a complete metric space and let $(a,b)$ be an open, possibly unbounded, interval. The curve $\mu:(a,b)\to \mathscr S$ is said to be $r$-\emph{absolutely continuous}, with $r\in[1,+\infty]$, if there exists a function $m\in L^r(a,b)$ such that
\begin{equation}\label{uppermetricderivative}
d(\mu(s),\mu(t))\leq \int_s^t m(\xi)d\xi.
\end{equation}
For any such curve, the limit
\begin{equation}\label{metricderivative}
|\mu'|(t):=\lim_{s\to t}\frac{d(\mu(s),\mu(t))}{|s-t|}
\end{equation}
exists and is termed \emph{metric derivative} of the curve $\mu.$

Let $g:(\mathscr S,d)\rightarrow [0,+\infty]$ and $\phi:(\mathscr S,d)\rightarrow \overline \R.$ The function $g$ is termed a \emph{strong upper gradient} for $\phi$ if, for every absolutely continuous curve $v:(a,b)\rightarrow (\mathscr S,d)$, the function $g\circ v$ is Borel measurable and we have
$$|\phi(v(t_0))-\phi(v(t_1))|\leq \int_{t_0}^{t_1} g(v(s))|v'|(s)ds,\;\;\;\forall a<t_0\leq t_1<b.$$ 

The (local) \emph{slope} of the functional $\phi$ at the point $\mu\in \Dom(\phi):=\{\nu\in\mathscr S:\phi(\nu)\in\R\}$ is defined by
$$|\partial \phi|(\mu):=\limsup_{\nu\to \mu}\frac{(\phi(\mu)-\phi(\nu))^+}{d(\mu,\nu)}.$$
In general, even under the $d$-lower semicontinuity assumption for $\phi$, the local slope $|\phi|$ is not a strong upper gradient, however some further assumptions (as certain type of convexity of $\phi$) imply that $|\partial \phi|$ is indeed a strong upper gradient. Precisely we have the following result, \cite[Cor. 2.4.10]{AmGiSa00}.
\begin{lemma}\label{Ambrosiolemma}
Assume that there exists $\lambda\in \R$ such that, for any $\mu_0,\mu_1\in \Dom(\phi)$ there exists a curve $\gamma:(0,1)\rightarrow \mathscr S$, $\gamma(0)=\mu_0,$ $\gamma(1)=\mu(1),$ satisfying the following convexity property for any $0<\tau<\frac{1}{\lambda^-}$.
\begin{align}
&\phi(\gamma(t))+\frac{d^2(\mu_0,\gamma(t))}{2\tau}\notag\\
\leq& t\phi(\mu_1)+(1-t)\phi(\mu_0)-\frac{t(t-\lambda\tau(1-t))}{2\tau} d^2(\mu_1,\mu_0),\;\;\forall t\in(0,1).\label{Lconvexitycondition}
\end{align}
Then $|\partial \phi|$ is $d$-lower semicontinuous and it is a strong upper gradient for $\phi.$
\end{lemma}
\begin{remark}\label{Ambrosioremark}
The easiest case for the application of Lemma \ref{Ambrosiolemma} is when for any two $\mu_0,\mu_1\in \Dom(\phi)$ there exists a curve $\gamma:(0,1)\rightarrow \mathscr S$, $\gamma(0)=\mu_0,$ $\gamma(1)=\mu(1),$ such that $\phi(\gamma(t))\leq t\phi(\gamma(1))+(1-t)\phi(\gamma(0))$ and the square of the distance of $\gamma(t)$ from $\mu_0$ is a $2$-convex function. That is
\begin{align*}
d^2(\mu_0,\gamma(t))\leq&(1-t)d^2(\mu_0,\gamma(0))) +t d^2(\mu_0,\gamma(1)))-\frac{2t(1-t)}{2}d^2(\mu_0,\mu_1)\\ 
=&t d^2(\mu_0,\mu_1)-t(1-t)d^2(\mu_0,\mu_1)\\
=&t^2d^2(\mu_0,\mu_1).
\end{align*}
Indeed, in such a case the condition \eqref{Lconvexitycondition} with $\lambda=0$ follows easily. 
\end{remark}

The curve $v:(a,b)\rightarrow (\mathscr S,d)$ is said to b a curve of maximal slope for the functional $\phi$ with respect to the upper gradient $g:(\mathscr S,d)\rightarrow [0,+\infty]$ if $\phi\circ v$ is $\mathscr L^1$-a.e. equivalent to a non-increasing map $\psi$ and we have
$$\psi'(t)\leq -\frac 1 2|v'|^2-\frac 1 2g^2(v(t)),\;\;\text{ for a.e. }t\in(a,b).$$
Note that, being $g$ an upper gradient, it follows that
$$ \psi'(t)=-|v'|^2=-g^2(v(t))=-g(v(t))|v'(t)|,\;\;\text{ for a.e. }t\in(a,b).$$

We list below some existence results for gradient flow equations. The following statement is a simplified and specialized version of \cite[Th. 2.3.3]{AmGiSa00}.
\begin{theorem}\label{Ambrosiotheorem}
Let $\phi$ be a $d$-lower semicontinuous functional bounded from below on $(\mathscr S,d).$ Assume that $|\partial \phi|$ is $d$-lower semicontinuous and is a strong upper gradient for $\phi.$ Then, if for any $\mu^0\in \Dom(\phi)$ the curve $[0,+\infty[\ni t\mapsto \mu(t)$ is in the class $GMM(\mu^0,\phi,d)$, then it is a curve of maximal slope for $\phi$ with respect to $|\partial \phi|.$ Moreover, for any such curve $[0,+\infty[\ni t\mapsto \mu(t),$ we have
\begin{equation}
\frac 1 2\int_0^t |\mu'|^2(s)ds+\frac 1 2\int_0^t |\partial \phi|^2(\mu(s))ds+\phi(\mu(t))=\phi(\mu^0).
\end{equation} 
\end{theorem}
The following result is contained in \cite[Th. 4.0.4]{AmGiSa00} and \cite[Th. 4.0.6]{AmGiSa00}.
\begin{theorem}\label{Ambrosiotheorem2}
Let $\phi$ be a $d$-lower semicontinuous functional bounded from below on $(\mathscr S,d).$ Assume that there exists $\lambda>0$ such that, for any $\mu_0,\mu_1,\nu\in \Dom(\phi)$ there exists a curve $\gamma:(0,1)\rightarrow \mathscr S$, $\gamma(0)=\mu_0,$ $\gamma(1)=\mu(1),$ satisfying the following convexity property for any $0<\tau<\frac{1}{\lambda^-}$.
\begin{align}
&\phi(\gamma(t))+\frac{d^2(\nu,\gamma(t))}{2\tau}\leq t\left[\phi(\mu_1)+\frac{d^2(\nu,\mu_1)}{2\tau}\right]\\
&\;\;\;\;\;\;+(1-t)\left[\phi(\mu_0)+\frac{d^2(\nu,\mu_0)}{2\tau}\right]-\frac{(1+\lambda\tau)t(1-t)}{2\tau} d^2(\mu_1,\mu_0),\;\;\forall t\in(0,1).
\end{align}
Then we have the following.
\begin{enumerate}
\item For any $\mu^0\in \Dom(\phi)$ the class $MM(\mu^0,\phi,d)$ contains a unique element $\mu(\cdot;\mu^0).$
\item The curve $\mu(\cdot;\mu^0)$ is a curve of maximal slope for the strong upper gradient $|\partial \phi|$ and it is locally Lipschitz.
\item The curve $\mu(\cdot;\mu^0)$ is the unique solution, among locally absolutely continuous curves $\nu$ such that $\lim_{t\to 0^+}\nu(t)=\mu^0,$ of the \emph{evolutional variational inequality}
\begin{equation}
\frac 1 2 \frac d{dt} d^2(\nu,\mu(t;\mu^0))+\frac1 2 d^2(\nu,\mu(t;\mu^0))+\phi(\mu(t;\mu^0))\leq \phi(\nu),\;\;\forall \nu\in \Dom(\phi) t>0.
\end{equation} 
\item If $\hat \mu\in \argmin \phi$, then we have
\begin{equation}\label{continuityalongtraj}
\phi(\mu(t;\mu^0))-\phi(\hat \mu)\leq \frac{d^2(\mu^0;\hat\mu)}{2t}.
\end{equation}
\item In particular, if the sublevels of $\phi$ are $d$-sequentially compact, the curve $\mu(\cdot;\mu^0)$ has a limit point $\bar \mu$ as $t\to +\infty$ and $\bar\mu\in \argmin \phi.$  
\end{enumerate}
\end{theorem}
\section{Some tools from functional analysis}\label{app2}
\subsection{A metric for the weak$^*$ topology}
In Section \ref{SecMinimizeE} we made a repeated use of the following result. We recall it here for the sake of completeness.
\begin{lemma}
The metric $d_w$ induces the weak$^*$ topology on the mass-bounded subsets of $\mathcal M^+(\Omega).$ 
\end{lemma}  
\begin{proof}
Let $\mu,\mu_j\in \mathcal M^+(\Omega)$ for any $j=1,2,\dots$ 
First we assume that $d_w(\mu,\mu_j)^2\to 0$ and we show that $\mu_j\rightharpoonup^*\mu.$ Indeed, being $d_w(\mu,\mu_j)^2$ a sum of non-negative terms, it follows that 
$$\left|\int \phi_k d\mu_j-\int \phi_k d\mu\right|\to 0 \text{ as }j\to +\infty,\;\;\forall k\in \N.$$
Hence, by density we get
$$\left|\int \phi d\mu_j-\int \phi d\mu\right|\to 0 \text{ as }j\to +\infty,\;\;\forall f\in \mathscr C^0_b(\Omega).$$
The last statement is one of the possible characterization of weak$^*$ convergence of measures.

On the other hand, if we assume that $\mu_j\rightharpoonup^*\mu$ and mass boundedness, i.e.,
$$ M:=\max\left(\sup_j\int d\mu_j,\int d\mu\right)^2<+\infty,$$
then we can prove that $d_w(\mu,\mu_j)^2\to 0$ as follows. First define
\begin{align*}
J(\epsilon,k)&:=\min\left\{j:\left|\int \phi_k d\mu_s-\int \phi_k d\mu\right|^2< \epsilon ,\forall s\geq j\right\},\;\;\forall \epsilon>0,\;\forall k\in \N,\\
j(\epsilon,N)&:=\max_{k\leq N}J(\epsilon,N),\;\;\forall \epsilon>0,\;\forall N,k\in \N.
\end{align*}
Then, for any $\epsilon>0$, pick $N\in \N$ such that $\sum_{k=N}^{+\infty}2^{-k}<\epsilon/(4M).$ So, $\forall j>j(\epsilon,N)$, we can write
\begin{align*}
d_w(\mu,\mu_j)^2&\leq \sum_{k=1}^N2^{-k}\left|\int \phi d\mu_j-\int \phi d\mu\right|^2+ \sum_{k=N+1}^{+\infty}2^{-k}\left|\int \phi d\mu_j-\int \phi d\mu\right|^2\\
&\leq \epsilon \sum_{k=1}^N2^{-k}+\sum_{k=N+1}^{+\infty}2^{-k}\|f\|_{L^\infty}^2\left(\int \phi d\mu_j+\int \phi d\mu\right)^2\\
&\leq \frac \epsilon 2 \sum_{k=1}^N2^{-k}+4M\frac \epsilon {4M} \sum_{k=N+1}^{+\infty}2^{-k}= \epsilon.
\end{align*}

\end{proof}

\subsection{Gamma convergence}\label{subsectiongammaconvergence}
The setting of our work considers elliptic PDEs with homogeneous Neumann boundary conditions. The functional tools that are used to analyze these type of problems are typically framed in the context of Dirichlet boundary conditions. For this reason we briefly recall here some relevant results and adapt them to our problem.

We briefly recall here the definition of $\Gamma$-convergence in topological spaces, that has been probably first introduced by De Giorgi \cite{DG75} in the framework of Calculus of Variation. We restrict our attention to its sequential characterization because it is the only one that is used in our proofs. We refer to \cite{Da93} for a comprehensive treatment of the subject.

Let $(X,\tau)$ a topological space and, for any $x\in X$, let us denote by $\mathcal N(x)$ the filter of the neighborhoods of $x.$ Let $f_j:X\rightarrow \overline{\R},$ $j\in \N.$ We define
\begin{align*}
\left(\gammaliminf_{j\to +\infty}f_j\right)(x)&:=\sup_{U\in \mathcal N(x)}\liminf_{j\to \infty}\inf_{y\in U}f_j(y),\\
\left(\gammalimsup_{j\to +\infty}f_j\right)(x)&:=\sup_{U\in \mathcal N(x)}\limsup_{j\to \infty}\inf_{y\in U}f_j(y).
\end{align*}
If there exists a function $f:X\rightarrow \R\cup\{-\infty,+\infty\}$ such that
\begin{equation}\label{gammadef}
\left(\gammalimsup_{j\to +\infty}f_j\right)(x)\leq f(x)\leq \left(\gammaliminf_{j\to +\infty}f_j\right)(x),\;\;\forall x\in X,
\end{equation}
then we say that $f_j$ $\Gamma$-converges to $f$ with respect to the topology $\tau$ and we write $f_j\gammaconverge f$ or $\gammalim f_j=f.$

Our main interest on this notion of convergence is given by the following property (cfr for instance \cite[Cor. 7.20]{Da93}). Assume that $f_j\gammaconverge f$ and $x_j$ is a minimizer of $f_j$. Then any cluster point $x$ of $\{x_j\}$ is a minimizer of $f$ and $f(x)=\limsup_j f_j(x_j).$ If moreover $x_j$ converges to $x$ in the topology $\tau,$ then $f(x)=\lim_j f_j(x_j).$

\subsection{Quadratic forms, elliptic operators, Neumann boundary conditions and $G$-convergence}\label{subsectionElliptic}
Let $\Omega$ be a bounded Lipschitz domain and $\lambda>0.$ We need to work with Neumann problems that formally can be written as
\begin{equation}\label{formalequation}
\begin{cases}
-\divergence((\mu+\lambda)\nabla u) = f& \text{ in }\Omega\\
\partial_n u=0& \text{ on }\partial \Omega\\
\int_\Omega u \,dx=0
\end{cases}
\end{equation}
for a non smooth scalar function $\mu$ and a given $f\in L^2(\Omega).$ We briefly recall, for the sake of giving a self-contained exposition, the definition of the linear operator associated to such a problem.

For any $\mu\in L_+^\infty(\Omega):=\{\mu\in L^\infty(\Omega), \mu\geq 0\text{ almost everywhere in }\Omega\}$ and $\lambda>0$ we can introduce the quadratic form $F_{\mu,\lambda}$ on the space of zero-mean square integrable functions
$$L^{2,0}(\Omega):=\left\{u\in L^2(\Omega):\int_\Omega u\,dx=0 \right\}$$
by setting
\begin{equation}\label{quadraticform}
F_{\mu,\lambda}(u):=\begin{cases}\int_\Omega (\mu+\lambda)|\nabla u|^2dx&\text{ if } u\in\mathscr C^1(\overline \Omega),\\
+\infty&\text{ otherwise}\end{cases}.
\end{equation}
We denote by $B_{\mu,\lambda}$ the bilinear form canonically associated to the quadratic form $F_{\mu,\lambda}.$
\begin{definition}\label{defAoperator}
Let $V_{\mu,\lambda}:=\clos_{L^2(\Omega)}{\Dom(F_{\mu,\lambda})}$, i.e., the closure of the domain of $F_{\mu,\lambda}$.  We can define the operator $A_{\mu,\lambda}$ as follows. First we set 
$$\Dom(A_{\mu,\lambda}):=\big\{u\in \Dom(F_{\mu,\lambda}): \exists f=:f(u)\in V_{\mu,\lambda}: B_{\mu,\lambda}(u,v)=\langle f;v\rangle_{L^2(\Omega)}\;\forall v\in \Dom(F_{\mu,\lambda})\big\}.$$
Then we let 
$$A_{\mu,\lambda}u=f(u),$$
where the uniqueness of such $f(u)$ follows by the density of $\Dom(F_{\mu,\lambda})$ in $V_{\mu,\lambda}.$
\end{definition}
\begin{definition}[Weak solutions]\label{solutiondefinition}
For any given $\mu\in L_+^\infty(\Omega),$ $\lambda>0$ and $f\in L^{2,0}(\Omega),$ we will refer to the \emph{solution} $u_{\mu,\lambda}$ of \eqref{formalequation} as the unique $u\in L^
{2,0}(\Omega)$ such that one of the following equivalent property holds
\begin{enumerate}
\item $u\in  \Dom(A_{\mu,\lambda})$ and $A_{\mu,\lambda}u = \pi_{\mu,\lambda} f$,
\item $u$ is the unique minimizer of the functional $F_{\mu,\lambda}(v)-2\langle f;v\rangle$ on $L^{2,0}(\Omega).$
\end{enumerate}
Here $\pi_{\mu,\lambda}$ denotes the $L^2(\Omega)$ orthogonal projection onto $V_{\mu,\lambda}.$
\end{definition}
The equivalence of the two formulations is a classical result found in most PDEs books. We refer to e.g. \cite[Prop. 12.12]{Da93}. Note in particular that we can write 
$$F_{\mu,\lambda}(u)-2\langle f;u\rangle=-\mathcal L_\lambda(\mu,u),$$
where $\mathcal L_\lambda$ has been defined in \eqref{Lldef}.

We also need to recall the definition of $G$-convergence of linear operators, first introduced by Spagnolo (see \cite{Sp67}, \cite{Sp68}) in the framework of homogenization problems.
\begin{definition}[$G$-convergence]\label{definitionGconv}
Let $c>0$ and let $S_c(X)$ denote the class of all self-adjoint operators $(A,\Dom(A))$ on the Hilbert space $X$ such that
$$\langle A u;u\rangle\geq c\|u\|_X^2.$$
We say that a sequence $\{A_j\}_{j\in\N}\subset S_c(X)$ $G$-converges in the strong (respectively weak) topology of $X$ to $A\in S_c(X)$ , if, for all $f\in X,$ we have that $A_j^{-1}\pi_j f\to A\pi f$ in the strong (respectively weak) topology of $X,$ where $\pi_j$ and $\pi$ denote the orthogonal projections onto $\overline{\Dom(A_j)}$ and $\overline{\Dom(A)},$ respectively. 
\end{definition}
In \cite{Sp67} it is shown that any family of linear second order uniformly elliptic operators (complemented by Dirichlet boundary conditions) is pre-compact with respect to the topology of $G$-convergence. This result can be slightly modified and specialized to our setting to prove the following \cite[Th. 20.3, Th. 22.10]{Da93}.
Let $\lambda>0$, let $\Omega\subset\R^n$ be a Lipschitz bounded domain and let $\mu_j$ be a bounded sequence in $L_+^\infty(\Omega)$. Then there exists a subsequence $k\mapsto\mu_{j_k}$ and a matrix valued function $M\in [L^\infty_+(\Omega)]^{n^2}$ such that $A_{\mu_{j_k},\lambda}$ $G$-converges (in the weak topology of $H^1(\Omega)$ and in the strong topology of $L^2(\Omega)$) to the operator $A_{M,\lambda}$ canonically associated to the quadratic form
$$F_{M,\lambda}(u):=\begin{cases}\int_\Omega \sum_{h,k=1}^n(M^{h,k}+\lambda\delta_{h,k})\partial_h u\partial_k u dx&\text{ if }u\in H^1(\Omega)\\
+\infty&\text{ otherwise}
\end{cases}.$$
\begin{remark}\label{pathologic}
It is worth recalling the well known fact that, even if each of the considered linear operators is isotropic (e.g., is defined by means of a scalar function $\mu_j$), its $G$-limit does not need to be isotropic. Perhaps even more surprisingly, in \cite{Sp68} the authors show that any non isotropic uniformly elliptic operator can be approximated in the topology of $G$-convergence by means of a sequence of isotropic operators. We will prevent such a phenomenon working on bounded subsets of $W^{1,p}_0(\Omega).$ Indeed, from any $W^{1,p}_0(\Omega)$ bounded sequence $\{\mu_j\}$ of nonnegative functions, we can extract an a.e. convergent subsequence to an a.e. non negative $L^\infty(\Omega)$ function $\mu.$ We can prove that, for any $\lambda>0$, the sequence of operators $A_{\mu_j,\lambda}$ $G$-converge to the operator $A_{\mu,\lambda}$, as stated in the following proposition.   
\end{remark}
\begin{proposition}\label{sigmaimpliesG}
Let $\mu_j$ be a sequence in $L^\infty_+(\Omega)$ $\sigma$-converging to $\mu\in L^\infty_+(\Omega).$ Then, for any $\lambda>0$, we have $A_{\mu_j,\lambda}\gconverge A_{\mu,\lambda}.$
\end{proposition}
\begin{proof}
For any fixed $\lambda>0$, the sequence of operators $\{A_{\mu_j,\lambda}\}_{j\in\N}$ is equi-bounded and equi-coercive. Due to \cite[Prop. 8.10]{Da93} (see also \cite[Rem. 20.5]{Da93}) the sequence of operators is $G$-converging to the operator associated to the point-wise limit of functions, that is $A_{\mu,\lambda}.$
\end{proof}

\begin{proposition}[Convergence of energy integrals; \cite{DGSp73}]\label{propconvergenceofenergies} Let $\mu_j$ be a sequence in $L^\infty_+(\Omega)$ $\sigma$-converging to $\mu\in L^\infty_+(\Omega).$ Then, for any $\lambda>0$ and denoting by $u_{\mu,\lambda}$ the weak solution 
$$
\begin{cases}
-\divergence{((\mu+\lambda)\nabla u)}=f&\text{ in }\Omega\\
\partial_n u=0&\text{ on }\partial \Omega\\
\int_\Omega udx=0
\end{cases},
$$
we have
\begin{equation}
\lim_j\int_B(\mu_j+\lambda)|\nabla u_{\mu_j,\lambda}|^2 dx=\int_B(\mu+\lambda)|\nabla u_{\mu,\lambda}|^2 dx,
\end{equation}
for all Borel subsets $B\subseteq\Omega$ such that $|\partial B|=0.$
\end{proposition}
\begin{proof}[Idea of the proof]
The result is very close to \cite[Th.22.10]{Da93}. Indeed a proof of our statement can be obtained by minor modifications of the proofs of \cite[Th. 22.10, Th. 21.3]{Da93}. Note that, due to Proposition \ref{sigmaimpliesG}, the $\sigma$-convergence of $\{\mu_j\}$ to $\mu$ implies the $G$-convergence of the operators $A_{\lambda,\mu_j}$ to $A_{\lambda,\mu}$. 
\end{proof}
\bibliographystyle{abbrv}
\bibliography{../biblio}
\end{document}